\documentclass{article}

\usepackage[english]{babel}
\usepackage{amsthm}
\usepackage[T1]{fontenc}
\usepackage[utf8]{inputenc}
\usepackage[english]{babel}
\usepackage{graphicx, color}
\usepackage{amssymb}
\usepackage{amsmath, xfrac}
\usepackage{latexsym}
\usepackage{caption, subcaption}
\usepackage{lineno}
\usepackage{comment}
\usepackage{mathtools}  
\usepackage{enumitem} 
\usepackage{hyperref}
\usepackage{cite}
\usepackage{float}
\usepackage{dsfont}
\usepackage{yhmath}
\usepackage{comment}

\newcounter{contSect} \numberwithin{contSect}{section}
 \numberwithin{contSub}{subsection}

\newtheorem{theorem}[contSect]{Theorem}
\newtheorem{corollary}[contSect]{Corollary}
\newtheorem{lemma}[contSect]{Lemma}
\newtheorem{proposition}[contSect]{Proposition}
\newtheorem{conjecture}[contSect]{Conjecture}
\newtheorem{claim}[contSect]{Claim}
\newtheorem{definition}[contSect]{Definition}
\newtheorem{observation}[contSect]{Observation}
\newtheorem{problem}{Problem}

\theoremstyle{remark}
\newtheorem*{remark}{Remark}
\DeclareMathOperator{\Cr}{cr}
\DeclareMathOperator{\Cp}{cp}

\usepackage[letterpaper,top=2cm,bottom=2cm,left=2.5cm,right=2.5cm,marginparwidth=1.75cm]{geometry}

\title{On the crossing profile of rectilinear drawings of $K_n$}
\author{Isaac Chen\footnote{Princeton Day School. \textit{Email:} isaac.chen.23@gmail.com} \text{ and} Oriol Solé-Pi\footnote{Department of Mathematics, Massachusetts Institute of Technology. \textit{Email}: oriolsp@mit.edu}}
\date{\today}

\begin{document}
\small\maketitle

\begin{abstract}
     We introduce the \textit{crossing profile} of a drawing of a graph. This is a sequence of integers whose $(k+1)^{\text{th}}$ entry counts the number of edges in the drawing which are involved in exactly $k$ crossings. The first and second entries of this sequence (which count uncrossed edges and edges with one crossing, respectively) have been studied by multiple authors. However, to the best of our knowledge, we are the first to consider the entire sequence. Most of our results concern crossing profiles of rectilinear drawings of the complete graph $K_n$. We show that for any $k\leq (n-2)^2/4$ there is such a drawing for which the $k^{\text{th}}$ entry of the crossing profile is of magnitude $\Omega(n)$. On the other hand, we prove that for any $k \geq 1$ and any sufficiently large $n$, the $k^{\text{th}}$ entry can also be made to be $0$. As our main result, we essentially characterize the asymptotic behavior of both the maximum and minimum values that the sum of the first $k$ entries of the crossing profile might achieve. Our proofs are elementary and rely mostly on geometric constructions and classical results from discrete geometry and geometric graph theory.
\end{abstract}

\section{Introduction}\label{sec:intro}

Graph drawings and their properties have been studied extensively, both from a theoretical and from a practical standpoint. Early research in this area was centered around \textit{planar graphs}, that is, graphs that can be drawn on the plane without incurring in any crossings between the edges; there many interesting questions regarding such graphs that are yet to be answered. Landmark results in the study of planar graphs include Euler's formula (which relates the number of nodes, edges and faces of a planar map), Kuratowski's theorem~\cite{kuratowski1930probleme,frink1930irreducible} (which provides a characterization of planar graphs in terms of forbidden minors), and the planar separator theorem of Lipton and Tarjan~\cite{lipton1979separator} (stating that every planar graph of order $n$ can be split into two large subgraphs by removing $O(\sqrt{n})$ nodes).

Going beyond planar graphs, one of the most natural questions one can ask is: what is the least number of crossings between edges that a drawing of a graph $G$ can have? This quantity is known as the \textit{crossing number of} $G$, and is denoted by $\Cr(G)$. Turán was the first to write about this when he posed a question (which is usually referred to as the brick factory problem) that can be stated as: What is the crossing number of the complete bipartite graph $K_{n,m}$~\cite{turan1977note}? This question has not yet been answered, and determining crossing numbers of graphs has in general proven to be a very difficult task. The literature surrounding crossing numbers is extensive, and we recommend the book by Schaefer~\cite{schaefersurvey} for an in depth treatment of the subject.  

The study of graph drawings and crossing numbers has had implications in the areas of VLSI (Very large scale integration)~\cite{VLSI,bhattleighton} and graph visualization~\cite{steele1998graph,tamassia2013handbook}. Large graphs are subjacent to many interesting phenomena that show up in all kinds of disciplines, and being able to produce visual representations of graphs which are easy to understand and convey relevant information has become a very important endeavor.

The purpose of this work is to study not the total number of crossings in a drawing, but how these crossings are distributed across the edges of the graph. More precisely, we are interested in bounding the number of edges which are involved in at most $k$ crossings and the number of edges involved in exactly $k$ crossings. As far as we know, before this paper, such questions had only really been considered in some particular cases and for very small values of $k$ (see sections~\ref{sec:e0} and~\ref{sec:crossing_sums}). We shall focus on the setting where we have a straight-line drawing of the complete graph $K_n$ with $n$ nodes (in a straight-line drawing, all edges are represented by segments that connect their endpoints), as it seems that questions of this kind are particularly rich in this case. As our main results, we present: \begin{itemize}

    \item A construction showcasing that for any $n$ and any $k\leq \frac{(n-2)^2}{4}$, the number of edges with exactly $k$ crossings might be of size linear in $n$ (Theorem~\ref{teo:lin e_k})\footnote{For $k > \left(\frac{n-2}{2}\right)^2$, it is impossible to have an edge with $k$ crossings in a straight-line drawing of $K_n$. Thus, this theorem applies to all $k$ for which an edge with $k$ crossings might exist.}.
    \item A construction showcasing that for any $k\geq 1$ and and sufficiently large $n\geq 1$, the number of edges with exactly $k$ crossings might be $0$ (Theorem~\ref{teo:0}). 
    \item Asymptotically tight upper and lower bounds for both the minimum and the maximum possible number of edges in such a drawing which are involved in at most $k$ crossings (theorems~\ref{teo:minS} and~\ref{teo:maxS}, respectively).
\end{itemize} Some related minor results are also discussed. We remark that any statement about rectilinear drawings of complete graphs can be interpreted as a statement about sets of points in general position on the plane. This motivates part of the discussion carried out in Section \ref{sec:final}.

All graphs considered in this paper are simple and finite. 

\subsection{Acknowledgments}
Most of the content of this paper was developed during the 2024 Research Science Institute (RSI) high school research program, where the first and second authors participated as student and mentor, respectively. We are grateful to Tanya Khovanova, Mihika Dusad, Victor Kolev, and an anonymous reviewer for their insightful feedback on the earliest versions of this write-up. We also thank David Jerison, Jonathan Bloom, and Tanya Khovanova for helping oversee the RSI research projects within the MIT Mathematics Department. The first author is grateful to Jenny Sendova for providing illuminating comments and endless encouragement. Finally, the authors thank RSI, the Center for Excellence in Education (CEE), and MIT for organizing and funding this research opportunity.

\section{Preliminaries and outline of the paper}\label{sec:prelim}
\subsection{Drawings and crossings}

We start off by defining drawings and rectilinear drawings of graphs.

\begin{definition}
Let $G$ be a graph. A \textit{drawing} $\mathcal D$ of $G$ is a representation of $G$ on the plane such that the vertices are mapped to distinct points and the edges are mapped to simple continuous curves that connect their respective endpoints. We further assume that no edge goes through a vertex other than at its endpoints, any two edges share at most one point\footnote[1]{Drawings where two edges are allowed to intersect multiple times are also interesting, and have been studied by multiple authors (see~\cite{fox_pach_2010,schaefersurvey} and the references therein). Removing the restriction that no two edges cross more than once would not have a significant impact on the results we present.} and are never tangent, and no three edges share an interior point.
\end{definition}

\begin{definition}
Given a graph $G$, a drawing of $G$ is said to be rectilinear if every edge is represented by a line segment that connects its endpoints.
\end{definition}

Next, we consider points other than vertices where two edges meet.

\begin{definition}
Let $\mathcal D$ be a drawing of a graph. A \textit{crossing} of $\mathcal D$ is a shared interior point of two distinct edges in $\mathcal D$.
\end{definition}

\begin{definition}
Let $G$ be a graph. The \textit{crossing number} $\Cr(G)$ is the minimum number of total crossings over all possible drawings of $G$. Similarly, the \textit{rectilinear crossing number} $\overline{\Cr}(G)$ is the minimum number of total crossings across all possible rectilinear drawings of $G$. 
\end{definition}

Note that $\overline{\Cr}(G)\geq \Cr(G).$ This inequality is strict for a wide variety of graphs (see \cite{cr4largecr} for an explicit example).

The following notable result, which is usually referred to as the crossing lemma, yields a lower bound on the crossing number in terms of the number of edges and nodes of the graph.

\begin{theorem}[Crossing lemma, Ajtai \textit{et al.} and Leighton,~\cite{crossinglemma, crossinglemma2}]
Let $G = (V, E)$ be a simple $n$-vertex graph. If $|E| \ge 4n$, then 
\[\Cr(G) \ge \frac{|E|^3}{64n^2}.\]
\end{theorem}

It follows immediately from this inequality that dense $n$-vertex graphs, i.e., graphs with $|E| = \Omega(n^2)$, have crossing number $\Omega(n^4)$. It is also easy to check that every $n$-vertex graph has crossing number at most $O(n^4)$. However, even the asymptotic behavior of ${\Cr}(K_n)$ and $\overline{\Cr}(K_n)$ as $n$ goes to infinity is not fully understood. The crossing and the rectilinear crossing number are possibly the two most widely studied parameters in geometric graph theory. The monograph by Schaefer~\cite{schaefersurvey} provides a comprehensive overview of the existing literature on crossing numbers of graphs.

\subsection{The crossing profile}

We wish to study how the crossings in a drawing are distributed amongst the edges of the gaph. Towards this end, we introduce the \textit{k-crossing index} and the \textit{crossing profile} of a drawing.

\begin{definition}
Let $\mathcal D$ be a drawing of a graph. For $k \ge 0$, define the \textit{$k$-crossing set} $E_k(D)$ as the set of edges in of $G$ that are part of exactly $k$ crossings in $\mathcal D$. Additionally, the \textit{$k$-crossing index} $e_k(\mathcal D) = |E_k(\mathcal D)|$ is the number of edges in $\mathcal D$ that are crossed $k$ times.
\end{definition}

\begin{definition}
Let $\mathcal D$ be a drawing of a graph. The \textit{crossing profile} of $\mathcal D$ is the sequence 
\begin{align}
\Cp(\mathcal D) = (e_0(\mathcal D), e_1(\mathcal D), \ldots). \nonumber
\end{align} 
\end{definition}

To the best of our knowledge, the crossing profile has not been studied in such generality before. However, it is clear that some classical problems in geometric graph theory can be phrased in terms of crossing profiles. Observe, for example, that a graph is $k$-planar\footnote{A graph is $k$\textit{-planar} if and only if it admits a drawing where each edge is involved in no more than $k$ crossings.} if and only if it admits a drawing $\mathcal D$ such that all non-zero terms of $\Cp(\mathcal D)$ are among its first $k+1$ entries.

Instead of analyzing the individual entries of the crossing profile, we can also consider the set of edges that are involved in at most $k$ crossings, which is equivalent to looking at the leftmost partial sums of this sequence.

\begin{definition}
Let $\mathcal D$ be a drawing of a graph. The \textit{$k$-crossing sum} $S_k(\mathcal D)$ is the number of edges in $\mathcal D$ that are crossed at most $k$ times. Equivalently, $S_k(\mathcal D) = \sum\limits_{i = 0}^{k} e_k(\mathcal D)$.
\end{definition}

For fixed $k$, we denote by $\max{e_k}(G)$ and $\min{e_k}(G)$ the maximum and minimum values, respectively, attained by $e_k(\mathcal D)$ across all drawings of $G$. Similarly, we denote by $\overline{\max}\ {e_k}(G)$ and $\overline{\min}\ {e_k}(G)$ the maximum and minimum values attained by $e_k({\mathcal{D}}(G))$ over all rectilinear drawings of $G$. The parameters $\max{S_k}(G)$, $\min{S_k}(G)$, $\overline{\max}\ {S_k}(G)$, and $\overline{\min}\ {S_k}(G)$ are defined analogously.

\subsection{Previous results regarding the \texorpdfstring{$0$}{TEXT}-crossing index}\label{sec:e0}

Previous explorations of $i$-crossing indices focused on the number of edges with $0$ crossings within both general and rectilinear drawings of complete graphs. In 1963, Ringel discovered a tight upper bound on $\max e_0(K_n)$.

\begin{theorem}[Ringel,~\cite{ringel1964extremal}]\label{Ringel}
For all $n \ge 4$, we have that $\max e_0(K_n) = 2n-2$.
\end{theorem}

As noted by Károlyi and Welzl~\cite{karolyi2001crossing}, $\overline{\max}\ e_0(K_n)$ is also $2n-2$. Figure~\ref{fig1} depicts a rectilinear drawing of $K_n$ attaining this value for $n = 7$. This construction can readily be extended to all $n \geq 4$.
\begin{figure}[H]
    \centering
    \includegraphics[scale=0.4]{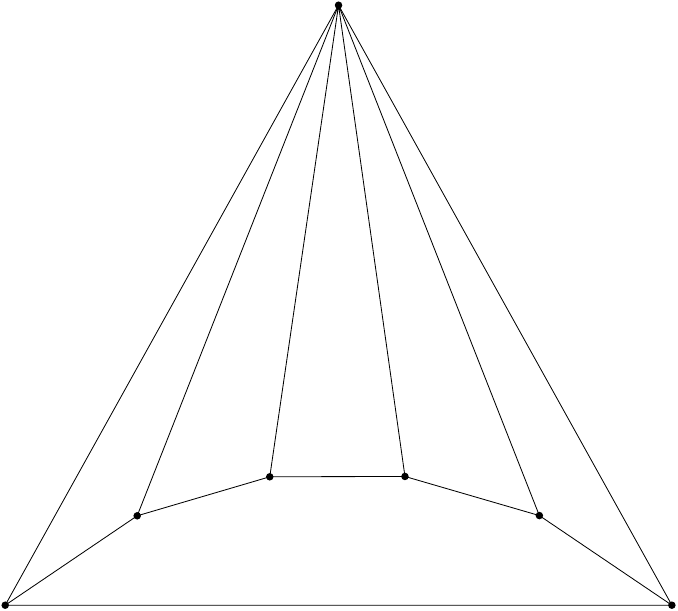}
    \caption{The positions of the vertices in a rectilinear drawing of $K_7$ with $12$ edges that are involved in no crossings.}
    \label{fig1}
\end{figure}

Looking at the other direction, Harborth and Mengersen found the exact values of $\min e_0(K_n)$ for all $n$.

\begin{theorem}[Harborth and Mengersen,~\cite{harborth1974edges}] 
The values of $\min e_0(K_n)$ for $n \ge 2$ are the ones displayed in Table~\ref{tbl:min(e_0)values}.
\end{theorem}

\setlength{\tabcolsep}{20pt}
\renewcommand{\arraystretch}{1.5}
\begin{table}[H]
\begin{center}
\begin{tabular}{|c|c|c|c|c|c|c|c|}\hline
n & $2$ & $3$ & $4$ & $5$ & $6$ & $7$ & $\mathbb{Z}_{\ge 8}$ \\ \hline
$\min e_0(K_n)$ & 1 & 3 & 4 & 4 & 3 & 2 & 0\\ \hline
\end{tabular}
\end{center}
\caption{Least possible number of edges with $0$ crossings in drawings of $K_n$.}
\label{tbl:min(e_0)values}
\end{table}

Moreover, they showed that as $\mathcal D$ ranges over all drawings of $K_n$,  $e_0(\mathcal{D})$ can attain every value between $\min e_0(K_n)$ and $\max e_0(K_n)$, except for the value $5$ when $n=4$ and the value $7$ when $n=5$.

In a separate paper, the value of $\overline{\min}\ e_0(K_n)$ was determined for $n \ge 3$.

\begin{theorem}[Harborth and Th{\"u}rmann, \cite{harborth1996numbers}]\label{Thurmann}
We have $\overline{\min}\ e_0(K_n) = 5$ for $n \ge 8$ and the values of $\min e_0(\overline{\mathcal{D}}_n)$ for $n \in [3, 7]$ are displayed in Table \ref{tbl:2nd}.
\end{theorem}

\setlength{\tabcolsep}{20pt}
\renewcommand{\arraystretch}{1.5}
\begin{table}[H]
\begin{center}
\begin{tabular}{|c|c|c|c|c|c|c|c|}\hline
$n$ & $3$ & $4$ & $5$ & $6$ & $7$ \\ \hline
$\overline{\min}\ e_0(K_n)$ & $3$ & $4$ & $5$ & $5$ & $6$ \\ \hline
\end{tabular}
\end{center}
\caption{Least possible number of edges with $0$ crossings in rectilinear drawings of $K_n$ for $n \in [3, 7]$.}
\label{tbl:2nd}
\end{table}

\subsection{On \texorpdfstring{$k$}{TEXT}-crossing sums}\label{sec:crossing_sums}

Here, we discuss some previous work related to $S_k(K_n)$. 

In addition to characterizing the possible values that $e_0(\mathcal D)$ can take for drawings of the complete graph, Harborth and Mengersen were able to determine the value of $\max S_1(K_n)$ up to an additive constant.

\begin{theorem}[Harborth and Mengersen,~\cite{harborth1990edges}]\label{teo:S1} 
For $n \ge 8$, we have 
\begin{align}
2n + \left\lfloor \frac{n-1}{2} \right\rfloor - 2 \le \max S_1(K_n) \le 2n + \left\lfloor \frac{n-1}{2} \right\rfloor + 7, \nonumber
\end{align}
and the exact values of $\max S_1(K_n)$ for $n \in [9]$ are displayed in Table~\ref{tbl:max(S_1)values}.
\end{theorem}

\setlength{\tabcolsep}{20pt}
\renewcommand{\arraystretch}{1.5}

\begin{table}[H]
\begin{center}
\begin{tabular}{|c|c|c|c|c|}\hline
$n$ & $1, 2, \ldots, 6$ & $7$ & $8$ & $9$ \\ \hline
$\max S_1(K_n)$ & $\binom{n}{2}$ & $18$ & $20$ & $22$ \\ \hline
\end{tabular}
\end{center}
\caption{Largest possible number of edges with $1$ crossing in drawings of $K_n$ for $n \in [9]$.}
\label{tbl:max(S_1)values}
\end{table}
 
Recall that a graph is $k$-planar if and only if it admits a drawing where each edge is involved in no more than $k$ crossings. The following result, due to Pach and Tóth, provides an upper bound on the number of edges in a $k$-planar graph.

\begin{theorem}[Pach and Tóth,~\cite{pach1997graphs}]\label{teo:k-planar}
Let $G = (V, E)$ be an $n$-vertex $k$-planar graph. For any $k \ge 1$, we have that
\begin{align}
|E| \le n \sqrt{16.875k}. \nonumber
\end{align}
\end{theorem}

This result immediately implies an upper bound for $S_k(\mathcal D)$.

\begin{proposition}\label{teo:nsqrtk}
Let $G$ be a simple graph on $n$ vertices. For any drawing $\mathcal D$ of $G$ and $k \ge 1$, we have 
\begin{align}
S_k(\mathcal D) = O(n \sqrt{k}). \nonumber
\end{align}
\end{proposition}

\begin{proof}
Consider the subgraph $H$ of $G$ that consists of all the edges involved in at most $k$ crossings of $\mathcal D$. We look at the drawing $\mathcal{D}(H)$ of $H$ induced by $\mathcal D$. Each edge in $\mathcal D(H)$ is crossed no more than $k$ times, so it follows that $H$ is a $k$-planar graph. Now Theorem~\ref{teo:k-planar} directly implies that $\mathcal{D}(H)$ has at most $n \sqrt{16.875k}$ edges or, equivalently, that $S_k(\mathcal D) \le n \sqrt{16.875k}$.
\end{proof}

As we shall see in Section~\ref{sec:maxS}, this result is tight up to a constant factor. We remark that the same sort of bound, but with a worse constant, can be obtained directly from the crossing lemma (without having to resort to the results in~\cite{pach1997graphs}).

\subsection{The cutting lemma}

The following result from computational and discrete geometry, which is commonly known as the cutting lemma, will be the main ingredient in our proof of the lower bound for $\overline{\min}\ S_k(K_n)$ that is to be presented in Section~\ref{sec:minS}.

\begin{theorem}[Cutting lemma, Matou{\v{s}}ek,~\cite{matouvsek1990construction,matousekcuttings1,matousekcuttings2}]\label{teo:cutting}
Let $S$ be a set of $n$ lines in $\mathbb{R}^2$ and $t \in (1, n)$ be a parameter. Then, $\mathbb{R}^2$ can be subdivided\footnote[1]{The regions are pairwise interior disjoint and cover all of $\mathbb{R}^2$.} into $r \le Ct^2$ generalized triangles (regions that are the intersection of three half-planes), where $C$ is an absolute constant, such that the interior of each generalized triangle is intersected by at most $\frac{n}{t}$ lines of $S$.
\end{theorem}

Higher dimensional analogs and various applications of this result can be found in~\cite{berg1995cuttings, chazelle2018cuttings}.

\subsection{Some additional notation and assumptions}

Throughout this paper, $n$ and $k$ denote non-negative integers. As is customary, we write $[n] = \{1, 2, \ldots, n\}$ for all $n \in \mathbb{Z}^+$. We always assume that the vertices of a drawing $\mathcal D$ are in general position, as sufficiently small perturbations can always be made to ensure this condition. Given two distinct points $A,B$ on the plane, $\overleftrightarrow{AB}$ will denote the line that goes through $A$ and $B$, while $\overline{AB}$ will denote the segment with endpoints $A$ and $B$. Following standard notation, if $f$ and $g$ are two functions of $t$ variables, we write $f(x_1,\ldots,x_t)\leq O(g(x_1,\ldots,x_t))$ (or, interchangeably, $f(x_1,\ldots,x_t)= O(g(x_1,\ldots,x_t))$) if there exists a constant $C>0$ such that $f(x_1,\ldots,x_t)\leq Cg(x_1,\ldots,x_t)$, and write $g(x_1,\ldots,x_t) \geq\Omega(f(x_1,\ldots,x_t))$ to indicate the same thing. We write $f(x_1,\ldots,x_t)=\Theta(g(x_1,\ldots,x_t))$ if $f(x_1,\ldots,x_t)=O(g(x_1,\ldots,x_t))$ and $f(x_1,\ldots,x_t)=\Omega(g(x_1,\ldots,x_t))$.

\subsection{Results and outline of the paper}
The goal of this paper is to study the behavior of $\overline\max\ e_k(K_n),\overline\min\ e_k(K_n),\overline\max\ S_k(K_n)$ and $\overline\min\ S_k(K_n)$. In Section~\ref{sec:maxe}, we showcase an intricate construction which implies that, for $0\leq k\leq \left\lfloor\frac{(n-2)^2}{4}\right\rfloor$, $\overline\max\ e_k(K_n)$ grows at least linearly in $n$. For $k=1$, we obtain a more precise bound of $\overline\max\ e_1(K_n)\geq \frac{3n}{2}-O(1).$ The purpose of Section~\ref{sec:mine} is to show that for all sufficiently large $n$, $\overline\min\ e_k(K_n)=0$ for all $k\geq1$. In Section~\ref{sec:maxS} we provide a construction which matches the $O(n\sqrt{k})$ upper bound for $\overline\max\ S_k(K_n)$ given by Proposition~\ref{teo:nsqrtk}, thus arriving at $\overline\max\ S_k(K_n) = \Theta(n \sqrt{k})$. In Section~\ref{sec:minS}, we prove what we consider to be our most interesting result. Namely, we determine the asymptotic behavior of $\overline\min\ S_k(K_n)$ up to (absolute) multiplicative constants for any $n$ and any $k$ between $1$ and $\left\lfloor \left( \frac{n-2}{2} \right)^2 \right\rfloor$, inclusive. Finally, we discuss some future research avenues and a related problem in Section~\ref{sec:final}.

\section{Regarding \texorpdfstring{$\overline{\max}\ e_k(K_n)$}{TEXT}} \label{sec:maxe}

In this section, we derive a lower bound for $\overline\max\ {e_k(K_n)}$ via a delicate construction and observe how an upper bound of $O(n\sqrt{k})$ for this same quantity follows directly from Proposition~\ref{teo:nsqrtk}. Then, in an attempt to obtain an analogue of Theorem~\ref{Ringel} for $k=1$ instead of $k=0$, we provide a better construction in this particular case. 

\subsection{A construction attaining \texorpdfstring{$\overline\max\ {e_k(K_n)} = \Omega(n)$}{TEXT}}

\begin{theorem}\label{teo:lin e_k}
For any $n$ and any $k$ between $0$ and $\left\lfloor \frac{(n-2)^2}{4} \right\rfloor$, inclusive, $\overline\max\ {e_k(K_n)} \geq \Omega(n)$.
\end{theorem}

\begin{remark}
If $P_1$ and $P_2$ are two vertices of a rectilinear drawing $\mathcal D$ of $K_n$ such that there are $m$ vertices on one side of the line $\overleftrightarrow{P_1P_2}$, then the number of edges crossing $\overline{P_1P_2}$ is at most $m((n-2) - m) \le \left( \frac{n-2}{2} \right)^2$. Thus, this theorem applies to all $k$ for which $e_k(\mathcal D) > 0$ is attainable.
\end{remark}

\begin{proof}
For $k=0$, the result follows from Theorem~\ref{Ringel}. Assume that $k>0$ and write $k = (m-2)^2 + r$ with $r \in [1, 2m-3]$, so that $m = \left\lceil \sqrt{k} \right\rceil + 1$. At a high level, our strategy consists of producing a configuration of either $2m-1$ or $2m$ points (similar to the one depicted in Figure~\ref{fig2}) which induces $\Omega(m)$ edges with $k$ crossings, and then arranging together many copies of it so that they do not interfere with each other to obtain a rectilinear drawing $\mathcal D$ of $K_n$ with $e_k(\mathcal D)=\Omega(n)$. The construction of these building blocks of size either $2m-1$ or $2m$ requires a significant amount of casework, most of which is deferred to the appendix. Here, we only include the proof for the case where both $n$ and $m$ are even, as it illustrates the important ideas while not being too long.

\begin{figure}[H]
\begin{center}
\includegraphics[scale = 0.7]{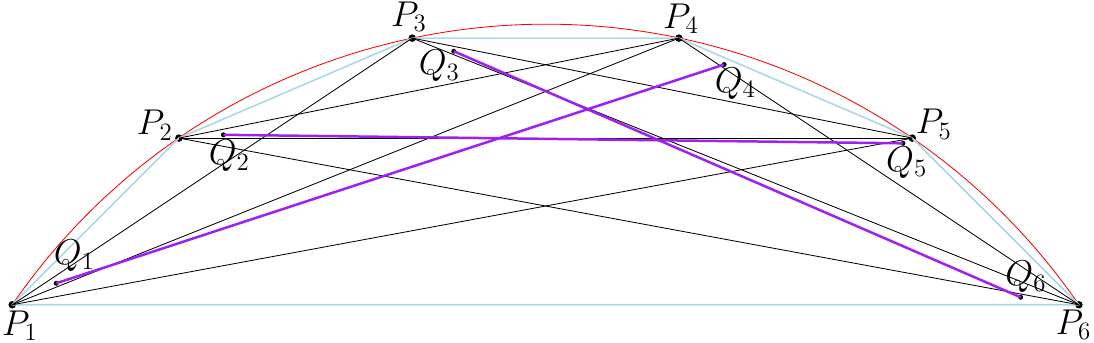}
  \caption{A configuration of $12$ points for $m=6$ and $k = 19$.}
  \label{fig2}
\end{center}
\end{figure}

The fact that $k\leq \left( \frac{n-2}{2} \right)^2$ implies $2m \le n$. Write $m = 2j$ and place $m$ equally spaced vertices $P_1, P_2, \ldots, P_m$ on a small circular arc in clockwise order. We also add $m$ other vertices $Q_1, Q_2, \ldots, Q_m$ such that $Q_i$ is very close to $P_i$ for $i \in [m]$. The precise positions of the $Q_i$'s will be specified below. Whenever we talk about $P_i$'s and $Q_i$'s, all indices are to be interpreted modulo $m$. Our goal is to place $Q_1, Q_2, \ldots, Q_m$ so that the $j$ \textit{interior diameters}
\begin{align}
\overline{Q_1Q_{j+1}}, \overline{Q_2Q_{j+2}}, \ldots, \overline{Q_jQ_{2j}} \nonumber
\end{align}
are each crossed exactly $k$ times. Notice that $\overline{Q_iQ_{i+j}}$ is crossed by all $(m-2)^2$ edges of the form $\overline{P_aP_b}$, $\overline{Q_aQ_b}$, $\overline{P_aQ_b}$, or $\overline{Q_aP_b}$ for $a \in [i+1, i + j - 1]$ and $b \in [i + j + 1, i + 2j-1]$. Now, we address the edges that are incident to at least one of $P_{i}$ and $P_{i+j}$.

We begin by perturbing $Q_1, Q_2, 
\ldots, Q_m$ so that $Q_i$ lies inside the angle $\angle P_{i+j}P_{i}P_{i+1}$ for all $i\in [m]$. This ensures that the $m-1$ rays
\begin{align}\label{equ:rays1}
\overrightarrow{P_iQ_{i+j-1}}, \overrightarrow{P_iP_{i+j-1}}, \overrightarrow{P_iQ_{i+j-2}}, \overrightarrow{P_iP_{i+j-2}}, \ldots, \overrightarrow{P_iQ_{i+1}}, \overrightarrow{P_iP_{i+1}}, \overrightarrow{P_{i-1}P_i}
\end{align}
appear in counterclockwise order with respect to $P_i$ and that the $m-1$ rays
\begin{align}\label{equ:rays2}
\overrightarrow{P_{i+j}Q_{i-1}}, \overrightarrow{P_{i+j}P_{i-1}}, \overrightarrow{P_{i+j}Q_{i-2}}, \overrightarrow{P_{i+j}P_{i-2}}, \ldots, \overrightarrow{P_{i+j}Q_{i+j+1}}, \overrightarrow{P_{i+j}P_{i+j+1}}, \overrightarrow{P_{i+j-1}P_{i+j}}
\end{align}
appear in counterclockwise order with respect to $P_{i+j}$. Moreover, because $Q_i$ and $Q_{i+j}$ lie on different sides of $\overline{P_iP_{i+j}}$, it follows that $\overline{P_iP_{i+j}}$ also crosses $\overline{Q_iQ_{i+j}}$. Now, we proceed with a more refined description of the positions of the $Q_i$'s.

If $r \in [1, m-1]$, then for every $i\in [j]$ we place  $Q_i$ so that it lies in between $\overrightarrow{P_iP_{i+j}}$ and $\overrightarrow{P_iQ_{i+j-1}}$, and we place $Q_{i+j}$ so that it lies in between the $(r-1)^{\text{th}}$ and $r^{\text{th}}$ rays listed in (\ref{equ:rays2}), where the $0^{\text{th}}$ ray is defined as $\overrightarrow{P_{i+j}P_i}$. This ensures that the edges incident to exactly one of $P_i$ and $P_{i+j}$ that cross $\overline{Q_iQ_{i+j}}$ are precisely those that correspond to the first $r-1$ rays listed in (\ref{equ:rays2}). Thus, the total number of edges crossing $\overline{Q_iQ_{i+j}}$ is indeed
\begin{align}
(m-2)^2 + 1 + (r-1) =  k. \nonumber
\end{align}
This way, each of the $j$ interior diameters is involved in precisely $k$ crossings. See Figure~\ref{fig2}.

If $r \in [m, 2m-3]$, then for every $i\in[j]$ we place $Q_i$ so that it lies in between the $(r-m+1)^{\text{th}}$ and $(r-m+2)^{\text{th}}$ rays of (\ref{equ:rays1}), and we place $Q_{i+j}$ so that it lies in between $\overrightarrow{P_{i+j}P_{i+j+1}}$ and $\overrightarrow{P_{i+j-1}P_{i+j}}$. This ensures that the edges incident to exactly one of $P_i$ and $P_{i+j}$ that cross $\overline{Q_iQ_{i+j}}$ are precisely those that correspond to the first $r-m+1$ rays of (\ref{equ:rays1}) and the first $m-2$ rays of (\ref{equ:rays2}). Thus, the total number of edges crossing $\overline{Q_iQ_{i+j}}$ is 
\begin{align}
(m-2)^2 + 1 + (r-m+1) + (m-2) = (m-2)^2 + r = k, \nonumber
\end{align}
so each of the interior diameters is crossed exactly $k$ times. See Figure~\ref{fig3}.

\begin{figure}[H]
\begin{center}
\includegraphics[scale = 0.7]{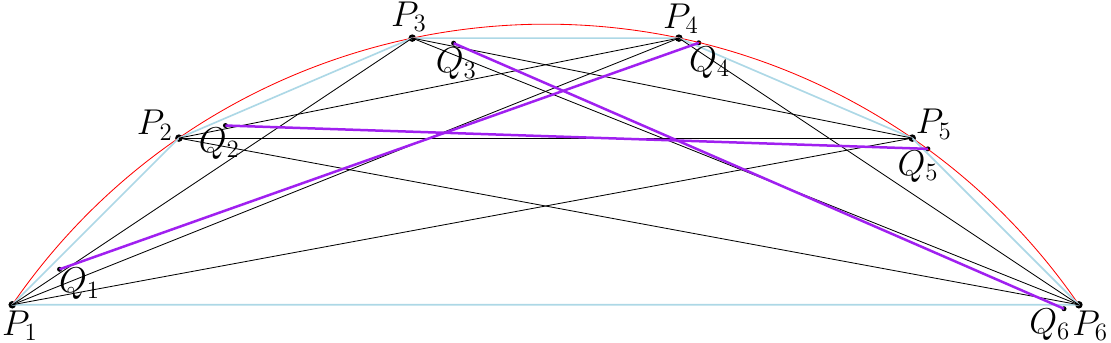}
  \caption{A configuration of $12$ points for $m=6$ and $k = 23$.}
  \label{fig3}
\end{center}
\end{figure}

We now describe how $t = \left\lfloor \frac{n}{2m} \right\rfloor$ copies of this configuration of $2m$ points can be arranged together so that no additional edges in the drawing interfere with the interior diameters of the individual configurations. Place points $A_1, B_1, A_2, B_2, \ldots, A_t, B_t$ on a circle $C$ in counterclockwise order. For each $\overline{A_iB_i}$, we construct a small circular arc $c_i$ containing $A_i$ and $B_i$ that is fully contained inside $C$, as seen in Figure~\ref{fig4}, and place the configuration of $2m$ points along this arc. Finally, place any of the at most $2m$ leftover vertices in a cluster near the center of $C$ to get a drawing $\mathcal D$ of $K_n$.

\begin{figure}[H]
\begin{center}
\includegraphics[scale = 0.4]{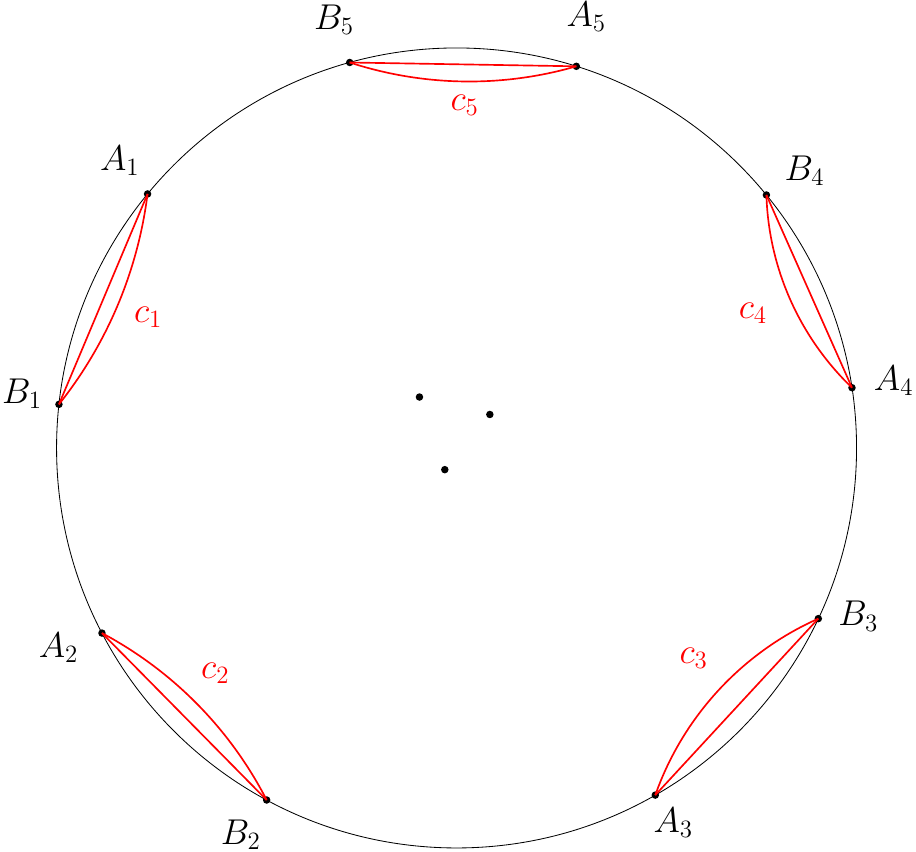}
  \caption{Grouping together $5$ copies of the $2m$-point configuration described above. Some leftover nodes have been placed near the center of the circle.}
  \label{fig4}
\end{center}
\end{figure}

We can make each $c_i$ sufficiently flat so as to ensure that no interior diameter of the configuration based at $c_i$ is intersected by additional edges of the drawing. Thus, each of these interior diameters still has $k$ crossings, and we have that
\begin{align}
e_k(\mathcal D) \ge t \cdot \Omega(m) = \left\lfloor \frac{n}{2m} \right\rfloor \cdot \Omega(m) = \Omega(n) \nonumber.
 \end{align}

\end{proof}

\subsection{An upper bound on \texorpdfstring{$e_k(\mathcal{D})$}{TEXT}}

The proposition below follows immediately from the fact that $e_k(\mathcal{D}) \le S_k(\mathcal{D}) \le n \sqrt{16.875k}$ for any drawing of an $n$-vertex graph, where the last inequality is given by Proposition~\ref{teo:nsqrtk}.

\begin{proposition}
Let $G$ be an $n$-vertex graph. For any drawing $\mathcal D$  of $G$ and any $k \ge 1$, it holds that $e_k(\mathcal{D}) = O(n \sqrt{k})$.
\end{proposition}

\begin{corollary}
    For every $n\geq 1$ and $k\geq 1$, $\overline\max\ e_k(K_n)= O(n\sqrt{k})$.
\end{corollary}

The issue of closing the gap between the lower and upper bounds for $\overline\max\ e_k(K_n)$ will be discussed further in Section \ref{sec:final}.

\subsection{About \texorpdfstring{$\overline\max\ {e_1(K_n)}$}{TEXT}}

Recall from Section~\ref{sec:e0} that $\overline\max\ {e_0(K_n)}=2n-2$. In contrast, we do not know the exact value of $\overline\max\ {e_1(K_n)}$, but we have the following bounds. 

\begin{claim}For $n \ge 8$,
\begin{align}
\left\lceil\frac{3n}{2} \right\rceil - 7 \le \overline\max\ {e_1(K_n)} \le 2n + \left\lfloor \frac{n-1}{2} \right\rfloor + 7. \nonumber
\end{align}
\end{claim}

\begin{proof}
First, note that Theorem~\ref{teo:S1} directly yields 
\begin{align}
\overline\max\ S_1(K_n) \leq \max S_1(K_n)\le 2n + \left\lfloor \frac{n-1}{2} \right\rfloor + 7. \nonumber
\end{align}

Next, we present a rectilinear drawing of $K_n$ attaining the claimed lower bound. We begin by handling the case where $n \equiv 0 \pmod{4}$, setting $n = 4m$. 

Place $m$ equally spaced vertices $P_1, P_2, \ldots, P_m$ on a sufficiently flat upwards-facing minor circular arc in clockwise order. Now, we place a vertex $A$ so that it lies far away from $\overline{P_1P_m}$ and is on the same side of $\overleftrightarrow{P_1P_m}$ as the rest of the $P_i$'s, and the projection of $A$ onto $\overline{P_1P_m}$ falls to the left of $P_1$. Next, add vertices $Q_1, Q_2, \ldots, Q_{m-1}$ so that $\overline{AQ_i}$ crosses and bisects $\overline{P_iP_{i+1}}$ and $Q_i$ lies sufficiently close to $\overline{P_iP_{i+1}}$ for $i \in [m-1]$. Finally, we add a line $l$ that lies parallel, below, and very close to $\overline{P_1P_m}$ and reflect all existing vertices with respect to $l$. We write $P_i \mapsto P_{i + m}$, $Q_i \mapsto Q_{i+m}$, and $A \mapsto B$ to denote the images of the original nodes under this reflection. See Figure~\ref{fig5}.

Under this setup, it is clear that $\overline{AQ_{i+m}}$ is the only edge crossing $\overline{P_iQ_i}$ for $i \in [m-1]$, $\overline{AP_{i+m}}$ is the only edge crossing $\overline{Q_{i-1}P_i}$ for $i \in [2, m]$, and $\overline{P_iP_{i+1}}$ is the only edge crossing $\overline{AQ_i}$ for $i \in [m-1]$. Symmetrically, we find $\overline{BQ_i}$ is the only edge that crosses $\overline{P_{i+m}Q_{i+m}}$ for $i \in [m-1]$, $\overline{BP_i}$ is the only edge that crosses $\overline{Q_{i+m-1}P_{i+m}}$ for $i \in [2, m]$, and $\overline{P_{i+m}P_{i+m+1}}$ is the only edge that crosses $\overline{BQ_{i+m}}$ for $i \in [m-1]$. This construction clearly gives $e_1(\overline{\mathcal{D}}_n) = 6m-6 = \frac{3}{2} \left(n - 4 \right)$, as required.

\begin{figure}[H]
\begin{center}
\includegraphics[scale = 0.56]{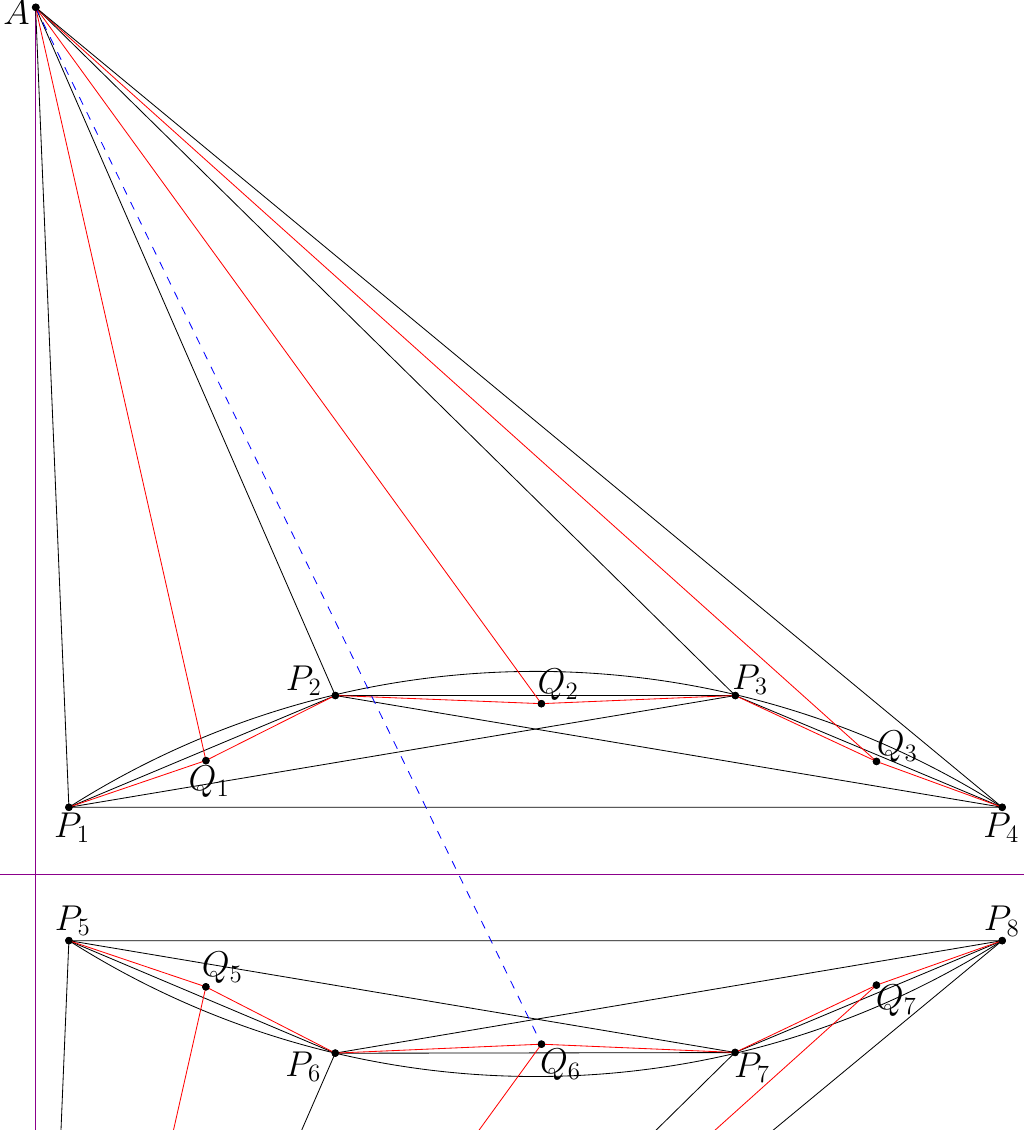}
  \caption{Construction for $n = 16$ where edges with exactly one crossing have been marked in red. Note, for example, that $\overline{AQ_6}$ is the only edge crossing $\overline{P_2Q_2}$.}
  \label{fig5}
\end{center}
\end{figure}

In the cases where $n = 4m-1$, $n = 4m-2$, or $n = 4m-3$, we delete $P_1$, $P_{m+1}$, and $Q_1$ in succession. It is easy to check that in each of these cases we get at least $\left\lceil\frac{3n}{2} \right\rceil - 7$ edges with exactly one crossing.
\end{proof}

We leave the issue of closing the gap between the upper and the lower bounds as an open problem.

\begin{problem}
 Determine $\overline\max\ {e_1(K_n)}$ for all $n$. Less ambitiously, does the limit $\lim_{n\rightarrow \infty}\overline\max\ \frac{e_1(K_n)}{n}$ exist? If so, what is its value? 
\end{problem}

\section{Drawings where no edges have exactly \texorpdfstring{$k$}{TEXT} crossings}\label{sec:mine}

The goal of this section is to prove the following theorem.

\begin{theorem}\label{teo:0}
    There exists a positive integer $N$ such that for any $k\geq 1$ and any $n\geq N$, $\overline\min\ e_k(K_n)=0$.
\end{theorem}

\begin{proof}
Let $m\geq 2$ be an integer. If the vertices of a rectilinear drawing $\mathcal D$ of $K_m$ are in convex position (i.e., they are the nodes of a convex polygon), then $e_k(\mathcal{D}) = 0$ for all $k$ not in $T_m$, where 
\begin{align}
T_m =\left\{a((m-2) - a) \mid a \in \left[1, 2,\ldots,  \left\lfloor \frac{m-2}{2} \right\rfloor \right] \right\}. \nonumber
\end{align}

For any two nonnegative integers $a,b$ with $a+b=n$, consider a rectilinear drawing $\mathcal D_{a,b}$ of $K_n$ induced by two flattened convex arcs facing each other and having $a$ and $b$ points each. Similarly, for any nonnegative integers $a,b,c$ with $a+b+c=n$, let $\mathcal D_{a,b,c}$ be a rectilinear drawing induced by three flattened convex arcs facing each other and having $a$, $b$ and $c$ points each. In each of these drawings, the arc containing $a$ points will be referred to as the \textit{upper arc}, and the one containing $b$ points will be the \textit{lower arc}. The nodes in the upper and lower arcs will be numbered from left to right, while the nodes in the third arc will always be numbered from top to bottom. See Figure~\ref{fig6} for an example showcasing $\mathcal D_{6,9,3}$.

\begin{figure}[H]
\begin{center}
\includegraphics[scale = 0.5]{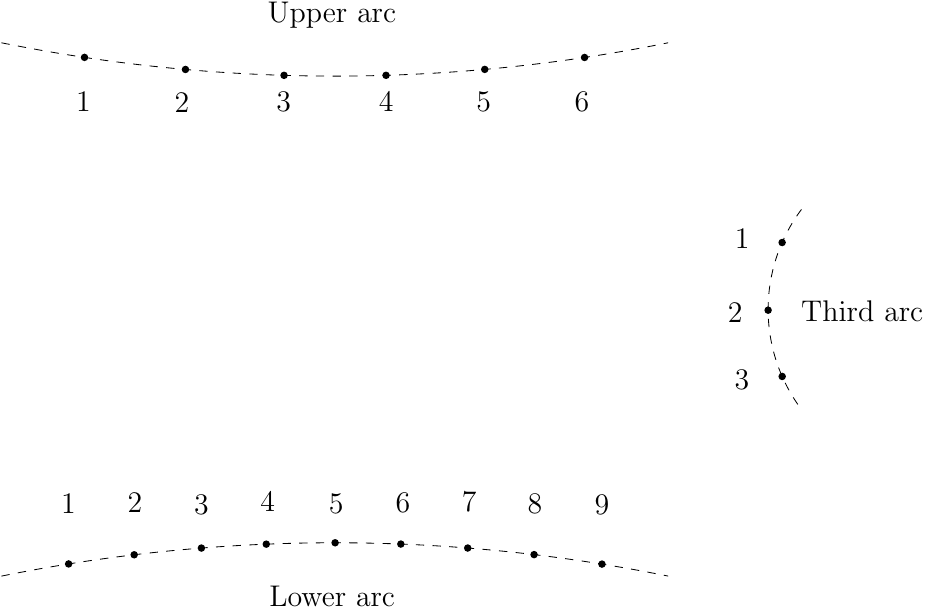}
  \caption{The drawing $\mathcal D_{6,9,3}$ with its vertices numbered.}
  \label{fig6}
\end{center}
\end{figure}

In $\mathcal D_{a,b}$, each edge going between the two arcs will be crossed by at most $(a-1)(b-1)$ other edges, while each edge inside one of the two arcs will be crossed by $t$ edges, where $t\in T_a\cup T_b$. As long as $n$ is not too small, for each $k$ in $\{1,2,\ldots,n\}$, one of $\mathcal D_{0,n}$, $\mathcal D_{1,n-1}$ contains no edges with exactly $k$ crossings. Henceforth, we assume that $k>n$. It is not hard to see that for any $k$ there are at most $\tau(k)$ positive integers $m$ such that $k\in T_m$, where $\tau(k)$ denotes the number of divisors of $k$. We will make use of an elementary fact from analytic number theory. Namely, for any $\epsilon>0$ there exists a positive integer $N_\epsilon$ such that $\tau(k)<k^\epsilon$ for all $k\geq N_\epsilon$~\cite{apostol2013introduction}. Set $\epsilon=1/10$ and take $N\geq N_\epsilon$. Consider the drawings $\mathcal D_{0,n},\mathcal D_{1,n-1},\ldots, \mathcal D_{\lceil n^{0.2}\rceil,n-\lceil n^{0.2}\rceil}.$ If $n\geq N$ then any $k$ with $n<k< n^2$ satisfies $\tau (k)< k^{0.1}< n^{0.2}$. Thus, for any such $k$, at least one of $T_{n},T_{n-1}\ldots,T_{n-\lceil n^{0.2}\rceil}$ does not contain it; let $T_j$ be any such set. Since $n-j\leq\lceil n^{0.2}\rceil$, it follows by bounding its largest element that $T_{n-j}$ also does not contain $k$ (since $k>n$). Furthermore, if $k\geq n^{1.2}$, then we also have that $k> (j-1)(n-j-1)$, so no edge in $\mathcal D_{j,n-j}$ has precisely $k$ crossings.

We are left to prove the statement for $n<k< n^{1+0.2}$, and for the rest of the proof we assume that this holds. Consider now the drawings of the form $\mathcal D(i,j)=\mathcal D_{\lfloor n/2\rfloor-i,\lceil n/2\rceil-j,i+j}$ for $n/100\leq i,j< n/50$. We will prove that at least one of the $n^2/1000$ $D(i,j)$'s has $0$ edges with exactly $k$ crossings. Let us analyze the different kinds of edges in these drawings.

\begin{itemize}
    \item If an edge connects two points in the same flattened convex arc, then the number of crossings it is involved in belongs to the set \[T_{\lfloor n/2\rfloor-i}\cup T_{\lceil n/2\rceil-j}\cup T_{i+j}.\] Hence, there are at most $3\tau(k)n/100\leq O(n^{1.2})$ pairs $(i,j)$ for which one of these edges has exactly $k$ crossings.

    \item If an edge connects the $m^{\text{th}}$ vertex in the upper arc to the $l^{\text{th}}$ vertex in the lower arc (see Figure~\ref{fig7}), then it crosses exactly \[(m-1)(\lceil n/2\rceil-j-l)+(l-1)(\lfloor n/2\rfloor-i-m) +(i+j)(m+l-2)\] other edges. In order for this quantity to be $k$, it must be the case that \[l+m=O\left(\frac{k}{n}\right),\] and that at least one of $l$ and $m$ is not equal to $1$. From here, we get that there are at most $O(k^2/n^2)=O(n^{0.4})$ pairs $(m,l)$ for which the edge has any chance of being involved in exactly $k$ crossings. Furthermore, once the pair $(m,l)$ is fixed, the number of crossings the edge is involved in is strictly increasing in $i$ as long as $m>1$, and it is strictly increasing in $j$ if $l>1$. It follows that there are at most $O((n(k/n)+n^{0.4})=O(n^{1.2})$ pairs $(i,j)$ for which an edge connecting the lower and upper arcs has exactly $k$ crossings.

    \item If the edge connects the $m^{\text{th}}$ vertex of the upper arc to the $l^{\text{th}}$ vertex from the third arc (see Figure~\ref{fig7}), then it crosses exactly \[(\lfloor n/2\rfloor-i-m)(\lceil n/2\rceil+i-l)+(l-1)(\lceil n/2\rceil-j+m-1)\] other edges. This time, in order for this quantity to equal $k$, we must have that \[(\lfloor n/2\rfloor-i-m)+(l-1)=O\left(\frac{k}{n}\right).\] Thus, there are at most $O(n^{0.4})$ for what the pair $(\lfloor n/2\rfloor-i-m,l-1)$ can be. If we fix such a pair so that $l-1\geq 1$, and then also fix $i$, then the above expression is strictly decreasing as $j$ increases. On the other hand, if after fixing the pair we have that $l-1=0$, then the expression does not depend on $j$ and is decreasing in $i$. Hence, there are at most $O(n^{0.4}\cdot n+n(k/n))=O(n^{1.4})$ pairs $(i,j)$ for which $\mathcal(i,j)$ contains an edge with $k$ crossings connecting the upper and third arcs.

    \item If the edge connects the $m^{\text{th}}$ vertex of the lower arc to the $l^{\text{th}}$ vertex from the third arc (see Figure~\ref{fig7}), then it crosses exactly \[(\lceil n/2\rceil-j-m)(\lceil n/2\rceil-i+l-1)+(i+j-l)(\lfloor n/2\rfloor-i+m-1)\] other edges. Now, if this equals $k$, then \[(\lceil n/2\rceil-j-m)+(i+j-l)=O\left(\frac{k}{n}\right).\] Thus, there are at most $O(n^{0.4})$ pairs $(\lceil n/2\rceil-j-m,i+j-l)$ for which the edge could have exactly $k$ crossings. If we fix such a pair such that $i+j-l>0$, and then fix $j$, the expression becomes strictly decreasing with $i$. If, after fixing the pair, $i+j-l=0$, then the expression is independent of $i$ and is decreasing in $j$. Just as above, there are at most $O(n^{1.4})$ pairs $(i,j)$ for which $\mathcal(i,j)$ contains an edge with $k$ crossings which connect the lower arc to the third arc.

\begin{figure}[H]
\begin{center}
\includegraphics[scale = 0.75]{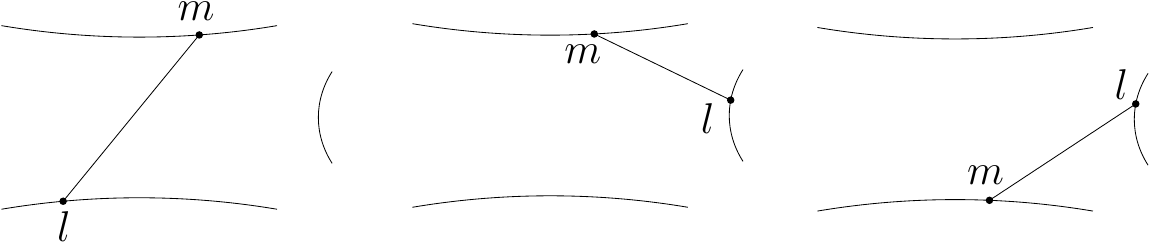}
  \caption{Left: An edge connecting the upper and lower arcs. Middle: An edge connecting the upper and third arcs.
  Right: An edge connecting the lower and third arcs.}
  \label{fig7}
\end{center}
\end{figure}

\end{itemize}

Putting everything together, we get that as long as $n$ is sufficiently large then at least one (in fact, most) of the drawings contains no edge involved in exactly $k$ crossings. This concludes the proof.
\end{proof}

\section{On \texorpdfstring{$\overline\max\ S_k(K_n)$}{TEXT}}\label{sec:maxS}

In this section we determine the asymptotic behavior of $\overline\max\ S_K(K_n)$ up to multiplicative constants.

\begin{theorem}\label{teo:maxS}
    For all $k$ with $1\leq k\leq \left( \frac{n-2}{2} \right)^2$, $\overline\max\ S_k(K_n) = \Theta(n \sqrt{k})$.
\end{theorem}

\begin{proof}
In light of Proposition~\ref{teo:nsqrtk}, it suffices to describe a rectilinear drawing $\mathcal D$ of $K_n$ with $S_k(\mathcal D)=\Omega(n\sqrt k).$

Let $m = 2 \lfloor \sqrt{k} \rfloor + 2$ and observe that $k \le \left( \frac{n-2}{2} \right)^2$ ensures $m \le n$. We choose $l = \lfloor \frac{n}{m} \rfloor$ congruent pairwise non-intersecting chords $a_1, a_2, \ldots, a_l$ of a circle $C$ such that they are spaced out evenly around $C$. For each $a_i$, erect a circular arc $c_i$ that touches the endpoints of $a_i$ and faces the center of the original circle, and let $R_i$ denote the finite region bounded by $a_i$ and $c_i$. Finally, add $m$ distinct vertices to each $c_i$ and place any leftover vertices in a cluster near the center of $C$.

We can clearly make each $c_i$ sufficiently flat so that the interior of $R_i$ is intersected only by edges with both vertices belonging to $c_i$ (see Figure~\ref{fig8}). Now, it is easy to see that any edge between two vertices on $c_i$ is crossed at most
\begin{align}
\left( \frac{m - 2}{2} \right)^2 = \lfloor \sqrt{k} \rfloor^2 \le k \nonumber
\end{align}
times, whence there are at least
\begin{align}
l \binom{m}{2} = \Omega(n \sqrt{k}) \nonumber
\end{align}
edges in $\mathcal{D}$ that are crossed at most $k$ times, as desired.
\end{proof}

\begin{figure}[H]
\begin{center}
\includegraphics[scale = 0.48]{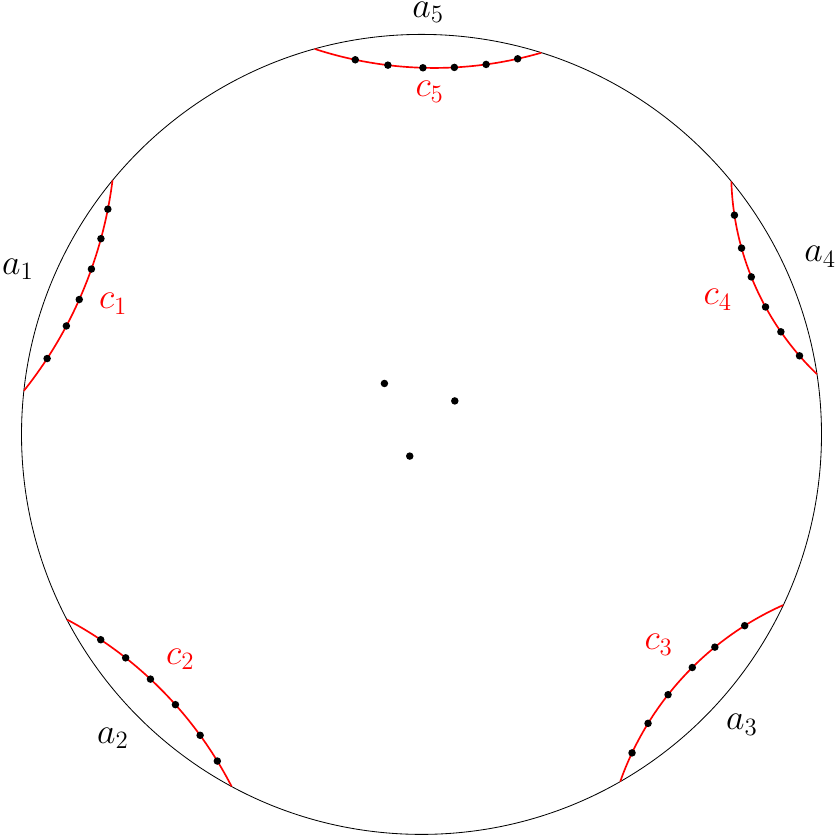}
  \caption{Construction attaining $S_k(\mathcal{D}) = \Omega(n \sqrt{k})$ for $k \in [4, 8]$, $m = 6$, and $n = 33$.}
  \label{fig8}
\end{center}
\end{figure}

\section{Matching upper and lower bounds for \texorpdfstring{$\overline\min\ S_k(K_n)$}{TEXT}}\label{sec:minS}

In this section we will prove the following result, which characterizes the asymptotic behavior of $\overline\min\ S_k(K_n)$ up to multiplicative constants.

\begin{theorem}\label{teo:minS}
    Let $n\geq 3$ and $k\leq\binom{n}{2}$. Then, we have that \[\overline\min\ S_k(K_n)=\begin{cases*}
    \Theta(1) & when $ k\leq n$,\\
    \Theta\left( \left( \frac{k}{n} \right)^2 \log{\left( \frac{k}{n} \right)} \right) & when $n<k\leq n^{3/2}$,\\
    \Theta\left( \left( \frac{k}{n} \right)^2 \log{\left( \frac{n^2}{k} \right)} \right) & when $n^{3/2}< k$.
\end{cases*} \]

\end{theorem}

In Subsection~\ref{sec:lower}, we show that $S_k(\mathcal{D}) = \Omega \left( \left( \frac{k}{n} \right)^2 \log{\left( \frac{k}{n} \right)} \right)$ holds when $n< k\leq n^{\frac{3}{2}}$, and $S_k(\mathcal{D}) = \Omega \left( \left( \frac{k}{n} \right)^2 \log{\left( \frac{n^2}{k} \right)} \right)$ holds when $n^{3/2}< k< n^2$. Finally, we show that both of these lower bounds are tight in Subsection~\ref{sec:upper}.

\subsection{Lower bound on \texorpdfstring{$S_k(\mathcal D)$}{TEXT}}\label{sec:lower}

\begin{theorem}\label{thm:lower S_k} 
Let $n\geq 3$ and consider a rectilinear drawing $\mathcal D$ of $K_n$. For  any $k$ with $n< k\leq n^{\frac{3}{2}}$, we have $S_k(\mathcal{D}) = \Omega \left(\frac{k^2}{n^2} \log \left(\frac{k}{n} \right) \right)$. For $n^{3/2}< k< n^2$, we have $S_k(\mathcal{D}) = \Omega \left(\frac{k^2}{n^2} \log \left(\frac{n^2}{k} \right) \right)$.
\end{theorem}

This theorem is a consequence of the following lemma, which provides a lower bound on the number of edges with at most $k$ crossings that are fully contained in a region of the plane bounded by two parallel lines.

\begin{lemma} \label{teo:block}
There exists a constant $C_1$ such that the following holds. Let $k > n \geq 2$ and let $\mathcal D$ be a rectilinear drawing of $K_n$ and take two parallel vertical lines $l_1$ and $l_2$ in the plane which don't go through any vertices of $\mathcal D$ and such that $l_1$ is to the left of $l_2$. Assume that the open region $R$ between $l_1$ and $l_2$ contains $m$ vertices of $\mathcal D$ for some $m$ with $\frac{2k}{n}< m \leq \min \left\{ \frac{1}{C_1} \left( \frac{k}{n} \right)^2, \frac{n}{2} 
\right\} $\footnote{It is possible that no such $m$ exists, in which case the theorem tells us nothing.}, and that at most $m$ vertices of $\mathcal{D}$ lie in the open half-plane to the left of $l_1$. Then, there are $\Omega \left( \left( \frac{k}{n} \right)^2 \right)$ edges which belong to $\bigcup\limits_{i = 0}^{k} E_i(\mathcal{D})$ and have both endpoints in $R$.
\end{lemma}

\begin{proof}

Let $L$ denote the set of edges in $\mathcal{D}$ that intersect the region $R$. Note that no edges outside of $L$ can intersect an edge fully contained in $R$. Let $L'$ be the set of lines in $\mathbb{R}^2$ formed by extending the edges in $L$. We claim that the number of edges which are fully contained in $R$ and are intersected by at most $k$ lines of $L'$ is $\Omega \left( \left( \frac{k}{n} \right)^2 \right)$.

In order to prove this, we begin by applying the cutting lemma (Theorem~\ref{teo:cutting}) to the set of lines $L'$ with parameter $t = \frac{|L'|}{k}$. First, we check that $t \in (1, |L'|)$ is satisfied. It is not difficult to realize that $|L'|=|L| \ge mn/2$, whence
\begin{align}
t = \frac{|L'|}{k} \ge \frac{mn}{2k} > 1. \nonumber
\end{align}
Furthermore, $t = \frac{|L'|}{k} <  |L'|$, as desired.

Let $T_1, T_2, \ldots, T_r$ denote the $r$ generalized triangles produced by the cutting lemma. The interior of $T_i$ is intersected by at most $\frac{|L'|}{t} = k$ lines of $L'$. Thus, if $P_i$ is the set of vertices of the drawing that belong to $R\cap T_i$, including the boundary of $T_i$, then any edge connecting two vertices of $P_i$ is intersected by at most $k$ lines of $L'$. Now, by adding over all indices and applying Jensen's Inequality, we arrive at the conclusion that the number of edges with the desired properties is at least
\begin{align}\label{equ:Jensen}
\sum_{i = 1}^{r} \binom{|P_i|}{2} \ge r \binom{\frac{1}{r} \sum\limits_{i = 1}^{r} |P_i|}{2} \ge r \binom{\frac{m}{r}}{2} = \frac{m(\frac{m}{r} - 1)}{2} \geq \frac{m^2}{4r}
\end{align}
where the last inequality follows from the fact that $\frac{m}{r} \ge 2$. To check that this holds, we first note that $|L'|\leq 2mn$. Moreover, the cutting lemma implies the existence of an an absolute constant $C_2$ such that
\begin{align}
r \le C_2t^2 = C_2 \left( \frac{|L'|}{k} \right)^2 \le C_2 \left( \frac{2mn}{k} \right)^2 = 4C_2m^2 \left(\frac{n}{k} \right)^2 \le 4C_2m^2 \left( \frac{1}{C_1m} \right) = \frac{4C_2m}{C_1}, \nonumber
\end{align}
so setting $C_1 \geq 8C_2$ guarantees $r \le \frac{m}{2}$.

Lastly, note that the right hand side of (\ref{equ:Jensen}) satisfies
\begin{align}
 \frac{m^2}{4r} \ge \frac{m^2}{4C_2t^2} = \frac{1}{4C_2} \left( \frac{mk}{|L'|} \right)^2 \ge  \frac{1}{4C_2} \left( \frac{mk}{2mn} \right)^2 = \frac{1}{16C_2} \left( \frac{k}{n} \right)^2,
\end{align}
so there are at least $\Omega \left( \left( \frac{k}{n} \right)^2 \right)$ edges with both endpoints in $R$ and with at most $k$ intersections with the lines in $L'$. Since any such edge must belong to $\bigcup\limits_{i = 0}^{k} E_i(\mathcal{D})$, this concludes the proof of the lemma.
\end{proof}

We are now ready to  prove Theorem~\ref{thm:lower S_k}.

\begin{proof}[Proof of Theorem~\ref{thm:lower S_k}] 

We start by partitioning the vertices of $\mathcal{D}$ into subsets using vertical parallel lines via a sweeping line process that starts on the left of all $n$ vertices and progresses rightwards (we assume that no two vertices lie in the same vertical line). Then, we obtain a lower bound on $S_k(\mathcal{D})$ by applying Lemma $\ref{teo:block}$ to each of these subsets and adding everything up.

If $k\leq n^{3/2}$, we set $r = \left\lfloor \log_{2} \left( \frac{k}{C_1n} \right) \right\rfloor$ and use $r$ vertical parallel lines to form $r$ consecutive disjoint regions containing $\frac{2k}{n}, \frac{4k}{n}, \ldots, \frac{2^{r}k}{n}$ vertices, in that order. First off, note this is possible, since \[\sum_{i=1}^r\frac{2^ik}{n}<2^{r+1}\frac{k}{n}\leq \frac{2k^2}{C_1n^2}\leq \frac{2}{C_1}n\] and we can assume that $C_1>2$. Furthermore, one can directly check that each of these $r$ regions satisfies the properties in the statement of Lemma~\ref{teo:block}, and thus fully contains $\Omega \left( \left( \frac{k}{n} \right)^2 \right)$ edges from $\bigcup\limits_{i = 0}^{k} E_i(\mathcal{D})$. Adding over all regions, we get that
\begin{align}
S_k(\mathcal{D}) \ge r \cdot \Omega \left( \left( \frac{k}{n} \right)^2 \right) \nonumber = \Omega \left( \left( \frac{k}{n} \right)^2 \log \left(\frac{k}{n} \right) \right), \nonumber
\end{align}
as desired.
 
If $k >n^{3/2}$, we set $r = \left\lfloor \log_{2} \left( \frac{n^2}{2k} \right) \right\rfloor$ and once again we use $r$ vertical parallel lines to construct $r$ consecutive disjoint regions containing $\frac{2k}{n}, \frac{4k}{n}, \ldots, \frac{2^{r}k}{n}$ vertices, in that order. This is possible, because \[\sum_{i=1}^r\frac{2^ik}{n}<2^{r+1}\frac{k}{n}\leq \frac{n^2}{k}\cdot\frac{k}{n}=n.\] By Lemma~\ref{teo:block}, within each of these $r$ regions lie $\Omega \left( \left( \frac{k}{n} \right)^2 \right)$ edges from $\bigcup\limits_{i = 0}^{k} E_i(\mathcal{D})$. As before, adding over all regions yields
\begin{align}
S_k(\mathcal{D}) \ge r \cdot \Omega \left( \left( \frac{k}{n} \right)^2 \right) \nonumber = \Omega \left( \left( \frac{k}{n} \right)^2 \log \left(\frac{n^2}{k} \right) \right), \nonumber
\end{align}
which completes the proof.
\end{proof}

\subsection{A construction matching the lower bound} \label{sec:upper}

\begin{theorem}\label{nestedtriangles}
Let $n\geq 3$. For $n<k\leq n^{3/2}$, there exists a rectilinear drawing $\mathcal{D}$ of $K_n$ attaining $S_k(\mathcal{D}) = O \left(\frac{k^2}{n^2} \log \left(\frac{k}{n} \right) \right)$. For $n^{3/2}<k<n^{2}$, there exists a rectilinear drawing $\mathcal{D}$ of $K_n$ attaining $S_k(\mathcal{D}) = O \left(\frac{k^2}{n^2} \log \left(\frac{n^2}{k} \right) \right)$.
\end{theorem}

\begin{remark}
For $k = O(n)$, the construction presented in Theorem~\ref{nestedtriangles} implies the existence of a drawing $\mathcal D$ with $S_k({\mathcal{D}}) = O(1)$. Furthermore, there are always at least $3$ edges involved in $0$ crossings (see Theorem~\ref{Thurmann}), so the regime $k \leq n$ in Theorem~\ref{teo:minS} is taken care of.
\end{remark}

\begin{proof}

We will assume that $n$ is a multiple of $3$, as it is not hard to extend our construction to the general case. Moreover, because the desired bound for this theorem is obvious when $n$ is bounded by a constant constant and also when $k = \Omega(n^2)$, we assume that $n\geq 9$ and $k = o(n^2)$ for the remainder of this proof.

Write $m =  \frac{n}{3}\geq 3$. Place the vertices of $\mathcal{D}$ in the plane so that they form $m$ concentric equilateral triangles $T_1, T_2, \ldots, T_m$ with center $O$ such that $T_{i+1}$ is the image of $T_i$ under a homothety at $O$ with scale factor $-\epsilon$, where $\epsilon > 0$ is sufficiently small (this will be made more precise later). Label the vertices of $T_i$ by $P_{i, 1}, P_{i, 2}, P_{i, 3}$, so that for $r \in \{1, 2, 3\}$, the points $P_{1, r},\ldots,P_{m,r}$ are all collinear (see Figure~\ref{fig9}). Later on, we will perform an elaborate perturbation on the vertices of $\mathcal{D}$ to achieve general position.

\begin{figure}[H]
\begin{center}
\includegraphics[scale = 0.57]{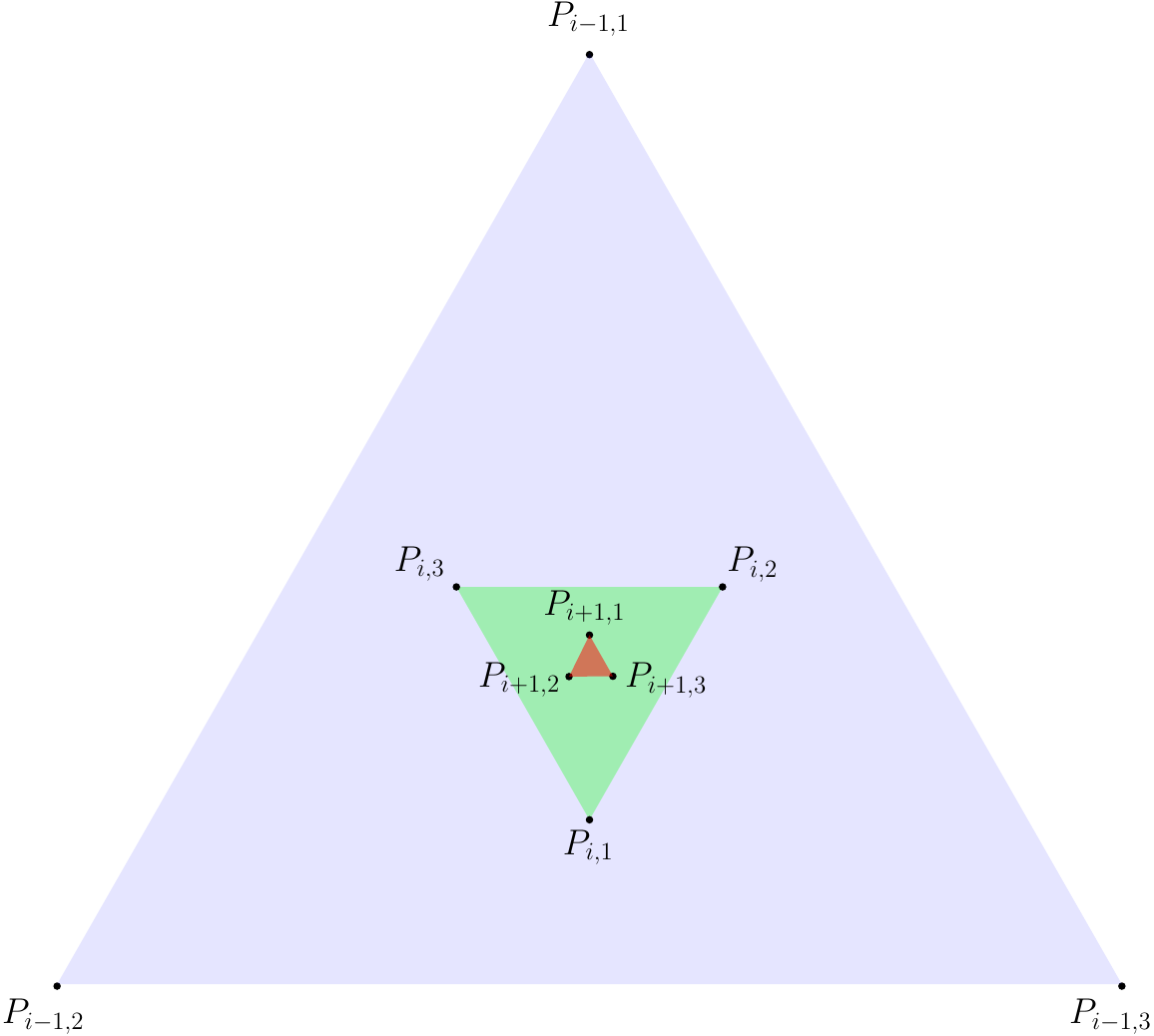}
  \caption{Three consecutive equilateral triangles $T_{i-1}$, $T_i$, and $T_{i+1}$.}
  \label{fig9}
\end{center}
\end{figure}

Our goal is to establish an upper bound on the number of edges with no more than $k$ crossings.  Consider an edge $e$ between distinct points $P_{i, r_1}$ and $P_{j, r_2}$. For the sake of convenience, we write $a=i-1$, $b=j-i-1$ and $c=m-j$. If $i - j$ is odd, then one can check that:
\begin{itemize}
    \item There are at least $\Omega(ab)$ edges that span from one of the $a$ triangles lying outside of $T_i$ to one of the $b$ triangles lying in between $T_i$ and $T_j$ and cross $e$.
    \item There are at least $\Omega(b^2)$ edges whose vertices both belong to one of the $b$ triangles lying in between $T_i$ and $T_j$ and cross $e$.
    \item There are at least $\Omega(bc)$ edges that span from one of the $b$ triangles lying in between $T_i$ and $T_j$ to one of the $c$ triangles lying inside of $T_i$ and cross $e$.
    \item There are at least $\Omega(ac)$ edges that span from one of the $a$ triangles lying outside of $T_i$ and $T_j$ to one of the $c$ triangles lying inside of $T_i$ and cross $e$.
\end{itemize}
If $i - j$ is even, then the first three conditions are still true, but the last one might fail to hold if we also have that $r_1 = r_2$. Henceforth, we say that $e$ is \textit{bad} if and only if $i - j$ is even and $r_1 = r_2$. Otherwise, we say that $e$ is \textit{good}.

Suppose that $e$ is involved in at most $k$ crossings. Since
\begin{align}
\Omega(b(m-2)) = \Omega(b(a+b+c)) = \Omega(ab) + \Omega(b^2) + \Omega(bc) \le k \nonumber,    
\end{align} 
we have that $b = \left( \frac{k}{m} \right)=O \left( \frac{k}{n} \right)$. Since $k=o(n)$, this implies that $b=o(n)$, and so $a+c =\Omega(n)$. Furthermore, if $e$ is a good edge, then $k\geq \Omega(ac)$, which in conjunction with the previous observation implies $\min\{a, c\} = O \left( \frac{k}{n} \right)$. From this, it readily follows that there are at most $O \left( \left( \frac{k}{n} \right)^2 \right)$ good edges with at most $k$ crossings.

Before attempting to obtain any sort of control on the number of bad edges with no more than $k$ crossings, we must specify how the points get perturbed to ensure general position. We describe only the perturbation of the points of the form $P_{x, 1}$ where $x \in [m]$ is odd, as the other cases can be handled analogously. Assume, without loss of generality, that the line that contains all points of the form $P_{x, 1}$ is vertical and that $P_{1,1}$ is the highest point in this line. 

Recalling that $b = O \left( \frac{k}{n} \right)$, we fix an absolute constant $C_3$ such that $b \le C_3 \left( \frac{k}{n} \right)$ holds unconditionally. Furthermore, set $C_3$ large enough so that $q = \left\lceil 4C_3 \left( \frac{k}{n} \right) \right\rceil >2$. Around a small neighborhood of the line that contains all $P_{x, 1}$'s, we construct $q$ disjoint clusters of vertical lines, each containing $\lceil\frac{n^2}{C_3k}\rceil$ lines. The lines inside each cluster are extremely close to each other. 

When perturbing the points, we never alter their position vertically, only horizontally. In fact, all perturbations will consist of pushing a point  either to the right or to the left until it lies on a specific line from the set described above. Let $t=\lceil m/2\rceil$ and, for each $d \in [t]$, write $d = lq + p$ with $p \in [0, q-1]$. We move $P_{2d-1, 1}$ so that it lies on the $(l+1)^{\text{th}}$ vertical line from the left of the $(p+1)^{\text{th}}$ cluster from the left. Because $\left\lceil \frac{t}{q} \right\rceil \le \left\lceil \frac{n^2}{C_3k} \right\rceil$, each of the $t$ vertices will end up on a different vertical line. For every positive integer $\ell$, define the $\ell^{\text{th}}$ \textit{layer} as the set of vertices $P_{2(\ell q-(q-1)) - 1, 1}, P_{2(\ell q - (q-2))-1, 1}, \ldots, \allowbreak P_{2(\ell q)-1, 1}$. 

Setting $i = 2i' - 1$ and $j = 2j' - 1$, consider the set $U$ of vertices in this grid-like configuration that lie below both $P_{i, 1}$ and $P_{j, 1}$ and also in between the vertical lines containing $P_{i, 1}$ and $P_{j, 1}$. The key geometric observation here is that if $\epsilon$ is sufficiently small (or, equivalently, if the vertical coordinates of the points we are perturbing shrink sufficiently fast), then $\overline{P_{i, 1}P_{j, 1}}$ is crossed by every single edge connecting a vertex in $U$ and a vertex that lies above $P_{i, 1}$. See Figure~\ref{fig10}.

\begin{figure}[H]
\begin{center}
\includegraphics[scale = 0.35]{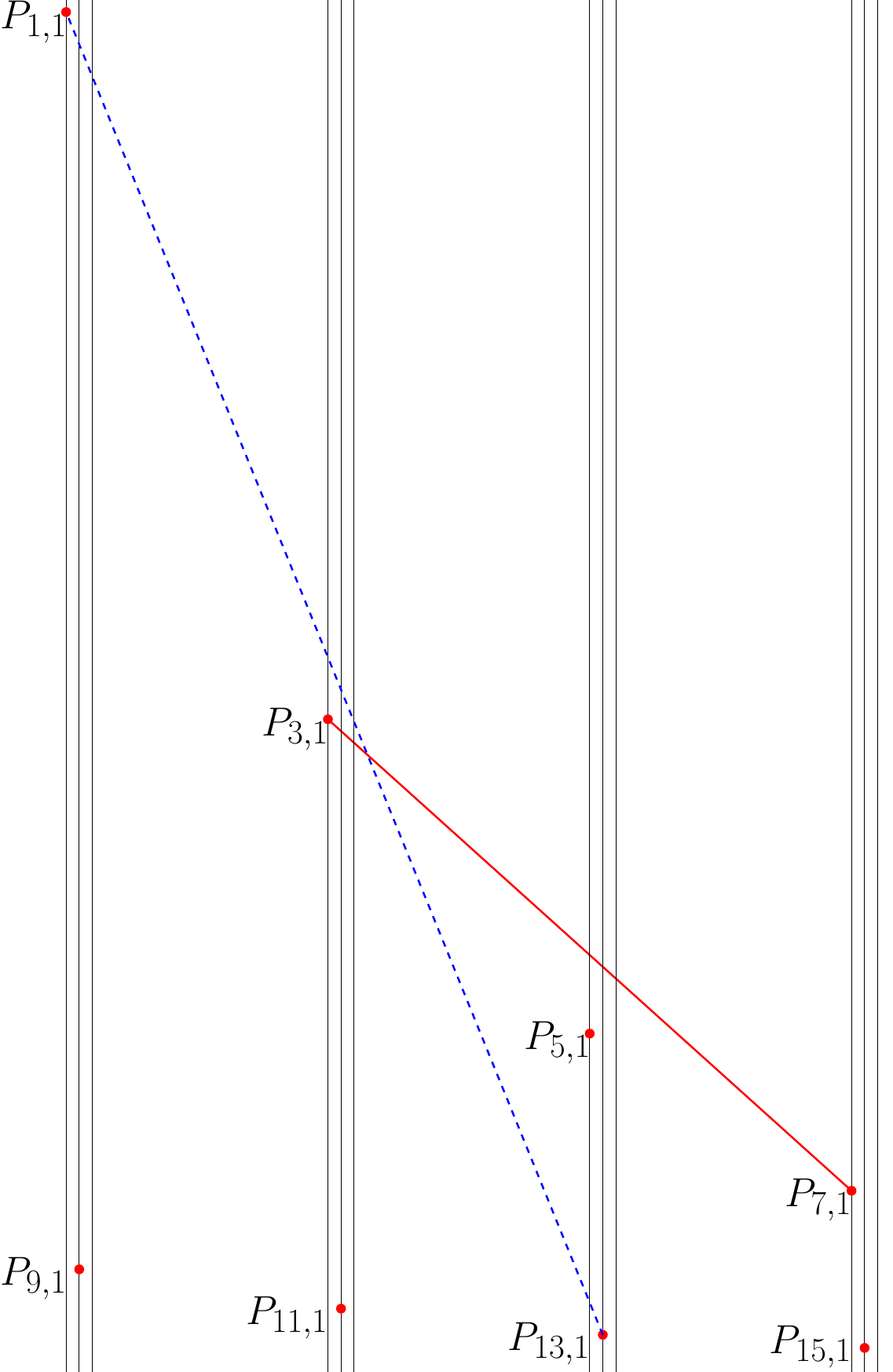}
  \caption{Sample perturbation with $4$ clusters each consisting of $3$ vertical lines. One instance of the key geometric observation for $\overline{P_{3, 1}P_{7, 1}}$ is also shown. If $\epsilon$ were to be made even smaller, eventually the segment $\overline{P_{1, 1}P_{11, 1}}$ would also cross $\overline{P_{3, 1}P_{7, 1}}$.}
  \label{fig10}
\end{center}
\end{figure}

We will say that an edge of the $t$-clique formed by $P_{1, 1}, P_{3, 1}, \ldots, P_{2t-1, 1}$ is \textit{isolated} if it crosses at most $k$ other edges of the $t$-clique. Using the above observation, we shall upper bound the number of isolated edges. Because all edges of the $t$-clique with at most $k$ crossings in $\mathcal D$ must be isolated, this will provide an upper bound on the number of bad edges with at most $k$ crossings. Henceforth, suppose $\overline{P_{i, 1}P_{j, 1}}$ is isolated.

\noindent\textbf{Case 1.}
Suppose that $i' = \Omega(n)$. In particular, this lets us assume $i' - 1 > 0$. Considering the set of edges between a vertex of $U$ and a vertex lying above $P_{i, 1}$ shows that $|U| \le \frac{k}{i'-1} = O \left( \frac{k}{n} \right)$.

Suppose that $P_{j, 1}$ belongs to the $\ell^{\text{th}}$ layer.
Notice that
\begin{align}\label{equ:BoundingVerticalClusters}
j - i - 1 = b \le C_3 \left( \frac{k}{n} \right) \le \frac{q}{4},
\end{align}
and $j'-i'-1=(j-i-2)/2< \frac{q}{8}$. Thus, since $P_{2j'-1, 1}$ belongs to the $\ell^{\text{th}}$ layer, then $P_{2i'-1, 1}$ belongs to the $\ell^{\text{th}}$ or $(\ell-1)^{\text{th}}$ layer. In the former case, there are clearly at least $j' - i' - 1$ clusters of vertical lines in between the clusters that $P_{i, 1}$ and $P_{j, 1}$ belong to, while for the latter case there are at least $q - \left( (j' - i' - 1) + 2 \right) > j'-i'-1$ clusters of vertical lines in between the clusters containing $P_{i, 1}$ and $P_{j, 1}$, where the inequality relies on $j' - i' - 1 <\frac{q}{8}$ and $q > 2$. Thus, we can always find $j'-i'-1$ clusters of vertical lines in between the clusters that $P_{i, 1}$ and $P_{j, 1}$ belong to. Let $\ell'$ denote the number of layers that lie completely below $P_{i, 1}$ and $P_{j, 1}$.

If $\ell'= 0$, then each of $i'$ and $j'$ can take at most $2q$ possible values, and we get that there are no more than $O(q^2)= O \left( \left( \frac{k}{n} \right)^2 \right)$ isolated edges $\overline{P_{i, 1}P_{j, 1}}$ of this form.

Suppose now that $\ell'\ge 1$. If $P_{i, 1}$ belongs to the $\ell^{\text{th}}$ layer, then $P_{2(i'+q) - 1, 1},$ $P_{2(i'+2q) - 1, 1}, \ldots, P_{2(i'+\ell' q) - 1, 1}$ all lie below $P_{i, 1}$ and $P_{j, 1}$ and in between the two vertical lines containing $P_{i, 1}$ and $P_{j, 1}$. If $P_{i, 1}$ belongs to the $(\ell-1)^{\text{th}}$ layer, then $P_{2(j'+q) - 1, 1}, P_{2(j'+2q) - 1, 1}, \ldots, \allowbreak P_{2(j'+\ell' q) - 1, 1}$ all lie below $P_{i, 1}$ and $P_{j, 1}$ and in between the two vertical lines that $P_{i, 1}$ and $P_{j, 1}$ belong to. Together with the $j'-i'-1$ clusters of lines that lie between $P_{i, 1}$ and $P_{j, 1}$, this yields $|U|\geq \Omega(\ell'(j'-i'))$.

Now, recalling that $|U| = O \left( \frac{k}{n} \right)$ gives $j'-i' = O \left( \frac{k}{\ell' n} \right)$, a deduction that yields two useful conclusions. Firstly, given that $P_{j, 1}$ lies in the $\ell^{\text{th}}$ layer, there are $q = O \left( \frac{k}{n} \right)$ possible values for $j'$, so the number of edges $\overline{P_{i, 1}P_{j, 1}}$ that are isolated is upper bounded by 
\begin{align}
O \left( \frac{k}{n} \cdot \frac{k}{\ell' n} \right) = O \left( \frac{1}{\ell'} \left( \frac{k}{n} \right)^2 \right). \nonumber
\end{align}
Secondly, it implies that
\begin{align}
\ell' = O \left( \frac{k}{n(j'-i')} \right) \le O \left( \frac{k}{n} \right), \nonumber
\end{align}
so we need only consider those cases where $\ell' \le O \left( \frac{k}{n} \right)$ or, equivalently $\ell\ge \left\lfloor \frac{t}{q} \right\rfloor - O \left( \frac{k}{n} \right)$.
 
On the other hand, the total number of layers is
\begin{align}
\left\lceil \frac{t}{q} \right\rceil \le \frac{t}{\left\lceil 4C_3\left( \frac{k}{n} \right) \right\rceil} = O \left( \frac{n^2}{k} \right). \nonumber
\end{align}
Thus, the number of isolated edges $\overline{P_{i, 1}P_{j, 1}}$ with $\ell'\ge 1$ is at most
\begin{align}
\sum_{\ell'=1}^{O\left(\min \left\{\frac{k}{n} , \frac{n^2}{k} \right\}\right)}  O \left( \frac{1}{\ell'} \left( \frac{k}{n} \right)^2 \right) 
&= O\left(\min \left\{ \left( \frac{k}{n} \right)^2 \log \left(\frac{k}{n} \right),  \left( \frac{k}{n} \right)^2 \log \left(\frac{n^2}{k} \right)  \right\}\right),
\nonumber
\end{align}
where the equality follows from evaluating a harmonic sum.

Overall, the number of isolated edges for this case is upper bounded by
\begin{align}
O \left( \left( \frac{k}{n} \right)^2 \right) +O\left(\min \left\{ \left( \frac{k}{n} \right)^2 \log \left(\frac{k}{n} \right),  \left( \frac{k}{n} \right)^2 \log \left(\frac{n^2}{k} \right)  \right\}\right) \\ = O\left(\min \left\{ \left( \frac{k}{n} \right)^2 \log \left(\frac{k}{n} \right),  \left( \frac{k}{n} \right)^2 \log \left(\frac{n^2}{k} \right)  \right\}\right).
\nonumber
\end{align}

\noindent\textbf{Case 2.} If $i' = o(n)$, the fact that $b = o(n)$ implies $c = \Omega(n)$. From $b=O\left(\frac{k}{n}\right)$, we also get that there are at most $O\left(\frac{k}{n}\right)$ isolated edges $\overline{P_{i, 1}P_{j, 1}}$ with $i'=1$, so we can assume that $i'>1$. As in Case $1$, we must have that $|U|\leq\frac{k}{i'-1}$. Also, if we let $\ell'$ stand once again for the number of layers that lie completely below $P_{1,j}$, then $|U|\geq \Omega(\ell'(j'-i'))$. Since $c = \Omega(n)$, we also know know that $\ell'\geq \Omega(\frac{n}{q})=\Omega(\frac{n^2}{k})$. Putting everything together, we arrive at \[\Omega\left(\frac{n^2}{k}(j'-i')\right)\leq \Omega\left(\ell'(j'-i')\right)\leq |U|\leq\frac{k}{i'-1}.\] Thus, there exists an absolute constant $C_4$ with \[j'-i'\leq C_4\left(\frac{k^2}{n^2(i'-1)}\right).\] In particular, this also tells us that if $i'-1>C_4\frac{k^2}{n^2}$, then there is no isolated edge involving $P_{1,i}$. Clearly, it is also true that $j'-i'<n$. 

All in all, we can conclude that the number of isolated edges with $i'>1$ is upper bounded by \[\sum_{i'=2}^{\min\left\{\left\lfloor C_4\frac{k^2}{n^2}\right\rfloor+1,n\right\}}\min\left\{C_4\left(\frac{k^2}{n^2(i'-1)}\right),n\right\}.\] A straightforward computation reveals that this sum evaluates to \[O\left(\min \left\{ \left( \frac{k}{n} \right)^2 \log \left(\frac{k}{n} \right),  \left( \frac{k}{n} \right)^2 \log \left(\frac{n^2}{k} \right)  \right\}\right),\] which concludes the proof of the Theorem.
\end{proof}

\section{Further research and connection to \texorpdfstring{$k$}{TEXT}-sets}\label{sec:final}

\subsection{Regarding \texorpdfstring{$\overline\max\ e_k(K_n)$}{TEXT}}

We have seen that, for all $n, 1\leq k\leq\left(\frac{n-2}{2}\right)^2$, \[\Omega(n)\leq \overline\max\ e_k(K_n)\leq O(n\sqrt{k}).\]

This is, unfortunately, all that we can say at the moment about the growth of $\overline\max\ e_k(K_n)$. We strongly believe that the upper bound can be improved on, and it wouldn't surprise us if it turns out that the lower bound is close to the truth.

\begin{conjecture}\label{conj:e_k}
    For every $n\geq 1$ and $k\geq 1$, $\overline\max\ e_k(K_n)= o(n\sqrt{k})$.\footnote{We are being somewhat imprecise here. What we actually mean is that for every $\epsilon>0$ there exists a constant $N$ such that if $n$ and $k$ are larger than $N$ then $\overline\max\ e_k(K_n)/(n\sqrt{k})<\epsilon$.}
\end{conjecture}

A proof of this conjecture, or a construction beating the $\Omega(n)$ lower bound, would be interesting even in special cases such as when $k=\Theta(n^2)$ or when $k$ is fixed and $n$ goes to infinity. As we saw in Section~\ref{sec:maxS}, there are rectilinear drawings attaining $\overline\max\ S_k(K_n)=\Theta(n\sqrt{k})$, so any argument that improves the upper bound for $e_k$ must be able to distinguish between edges with $k$ crossings and edges with less than $k$ crossings.

Along similar lines, one could ask whether $\Omega(n^2)$ edges might be concentrated on a small interval of the crossing profile. For example, is it true that for every $n$ there exists a rectilinear drawing $\mathcal D$ and two positive integers $m$ and $\ell$ such that $\ell=o(n^2)$ and $\sum_{k=m}^{m+\ell}S_k(\mathcal D)=\Omega(n^2)$? As of right now, all that we know is that an interval of length $\ell$ can encompass up to $\Omega(n\sqrt{\ell})$ edges, as shown by the drawing where the vertices are in convex position.

\subsection{On \texorpdfstring{$k$}{TEXT}-edges, \texorpdfstring{$e_k$}{TEXT}, and another variant of our problem}

Below, we draw an analogy between the problem of bounding $\overline{\max}\ e_k(K_n)$ and a very well known question of Lovász~\cite{lovasz1971number} regarding the number of $k$-edges induced by a set of points.

Give a set $P$ of $n$ points in general position on the plane, a $k$\textit{-edge} is a pair of points $A,B\in P$ such that on one of the two open half-planes defined by the line $\overrightarrow{AB}$ contains exactly $k$ points of $P$. Similarly, a $\leq k$\textit{-edge} consists of a pair of points in $P$ for which one of these two open half-planes contains at most $k$ elements of $P$. The maximum number of $k$-edges and $\leq k$-edges that a set of $n$ points can have has been widely studied. In a landmark paper, Dey~\cite{dey1998improved} showed that the maximum number of $k$-edges is at most $O(nk^{1/3})$, and improving this result is considered a major open problem in discrete geometry (the best known constructions achieve $n2^{\Omega(\sqrt {\log k})}$ ~\cite{nivasch2008improved}). On the other hand, the maximum possible number of $\leq k$ edges is known to be of order $\Theta(nk)$~\cite{alon1986number}, and the least possible number of $k$ edges has also been studied~\cite{k-setslovasz}.

For any set of points $P$ as above, consider the rectilinear drawing $\mathcal D_P$ of $K_n$ where the nodes correspond to the elements of $P$. In this setting, a $k$-edge can also be defined as an edge which, after being extended to a line, crosses exactly $k$ other edges in the drawing. In this paper, we have dealt with the number of edges which cross exactly (or at most) $k$ other edges. One could also study the number of edges whose interior intersects the extensions of exactly (or at most) $k$ other edges in the drawing.

Given a rectilinear drawing $\mathcal D$ of a graph $G$, let $E_k'(\mathcal D)$ denote the set of edges in the drawing whose interior is intersected by exactly $k$ lines that arise from extending an edge of the drawing infinitely in both directions, and write $e_k'(\mathcal D)=|E_k'(\mathcal D)|$. We define $S_k'(\mathcal D)$, $\overline{\max{}}e_k'(G)$, $\overline{\min{}}e_k'(G)$, $\overline{\max{}}S_k'(G)$, $\overline{\max{}}S_k'(G)$ just as we did for the unprimed versions.

\begin{theorem}\label{teo:S'}
    Let $n\geq 1$ and $k\leq\binom{n}{2}$. Then, $\overline\max\ S_k'(K_n)=\Theta(n+k)$ and there exist two absolute positive constants $c$ and $C$ such that \[\overline\min\ S_k'(K_n)=\begin{cases*}
    0 & if $ k\leq cn^{3/2}$,\\
    \Theta\left( \left( \frac{k}{n} \right)^2  \right) & if $Cn^{3/2}< k$.
\end{cases*} \]

\end{theorem}

\begin{proof} We begin by addressing the growth rate of $\overline\max\ S_k'(K_n)$.

    Any set of $n$ points in convex position induces a drawing showcasing that $\overline\max\ S_k'(K_n)\geq\Omega(n+k)$. To show that $S_k'$ cannot be larger than this, it suffices to prove that in a rectilinear drawing of $K_n$ every vertex is incident to no more than $O((n+k)/n)$ edges of $E_k'(\mathcal D)$. We require the following geometric observation (see Figure~\ref{fig11}).

\begin{observation}
    Let $X$, $Y$ and $Z$ be three points on the plane, not all in the same line. Then, for any other point $W$ such that $Y$ and $Z$ lie on the same side of the line $\overleftrightarrow{WX}$, either $\overleftrightarrow{WY}$ intersects $\overline{XZ}$ or $\overleftrightarrow{WZ}$ intersects $\overline{XY}$.
\end{observation}

\begin{figure}[H]
\begin{center}
\includegraphics[scale = 0.5]{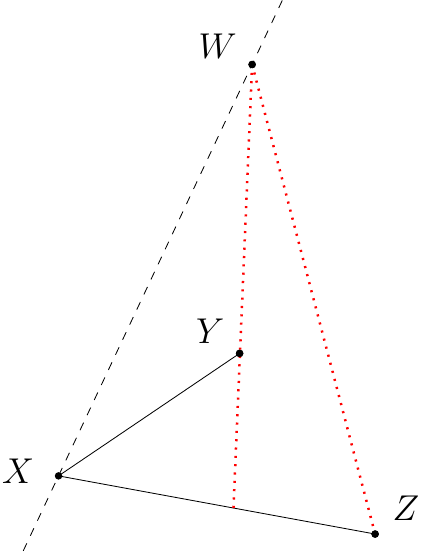}
  \caption{One of the dotted red lines must cross the interior of one of the two solid segments.}
  \label{fig11}
\end{center}
\end{figure}

Let $A$ be a point representing a vertex of $K_n$ in a rectilinear drawing $\mathcal D$. Now, let $d(A)$ denote the number of edges which are incident to $A$ and are intersected by at most $k$ lines spanned by the edges of the drawing, and consider the endpoints (distinct from $A$) $B_1$, $B_2,\ldots,$ $B_{d(A)}$ of these edges. For any other point $Q\neq A$ which also represents a node in the drawing, the observation above implies that if $Q\not\in\{ B_i,B_j\}$ and $B_i, B_j$ lie on the same side of the line $\overleftrightarrow{QA}$, then either $\overleftrightarrow{QB_i}$ intersects the segment $\overline{AB_j}$ or $\overleftrightarrow{QB_j}$ intersects the segment $\overline{AB_i}$. Adding over all valid pairs $B_i$ $B_j$ and using Jensen's inequality, we get that the lines joining $Q$ to each of the $B_i$'s have a total of at least \[2\binom{(d(A)-1)/2}{2}\] intersection points with the interiors of the segments $\overline{AB_1},\ldots,\overline{AB_{d(A)}}$. Adding over all nodes $Q$ other than $A$ and dividing by $2$ to make up for possible double counting, we get that interiors of the segments $\overline{AB_1},\ldots,\overline{AB_{d(A)}}$ are involved in a total of at least \[(n-1)\binom{(d(A)-1)/2}{2}\] intersections with the extensions of the edges of $\mathcal D$. On the other hand, we know that this quantity is at most $kd(A)$, so we get \[(n-1)\binom{(d(A)-1)/2}{2}\leq kd(A).\] From here, it is not hard to deduce that $d(A)\leq O((n+k)/n)$, as required.

Next, we show that $\overline\min\ S_k'(K_n)$ grows as stated above.

For the upper bound, it suffices to provide a construction which attains it. The situation here turns out to be simpler than in the proof of Theorem \ref{teo:minS}, and it suffices to take a rectilinear drawing $\mathcal D$ of $K_n$ induced by perturbing the vertices of a $\sqrt{n}\times\sqrt{n}$ grid\footnote{If $n$ is not a perfect square, we can instead take a $\lceil \sqrt{n}\rceil\times\lceil \sqrt{n}\rceil$ and then delete a subset of the nodes so that $n$ of them remain.}. The key fact about this drawing is that any edge of (Euclidean) length $d$ intersects $\Theta(dn^{3/2})$ of the extensions of the other edges in the drawing. The proof of this fact is rather straightforward but somewhat cumbersome, and will not be included here. From this, we get that:

\begin{itemize}
    \item If $k\leq cn^{2/3}$ and $c$ is small enough, then $E_k'(\mathcal D)=0$.
    
    \item If $Cn^{2/3}\leq k$ and $C$ is large enough, then $S_k'(\mathcal D)$ is at most the number of edges of length $O(k/n^{3/2})$. Each nodes is incident to $O(k^2/n^3)$ edges of this length, so the result follows. 
\end{itemize} 

For the lower bound, assume that $Cn^{3/2}<k$ and consider a drawing of $K_n$. We can repeat the argument from the proof of Lemma~\ref{teo:block}, but with the whole set of points instead of a strip that contains only $m$ of them, to get the desired bound. This works, since $\frac{2k}{n}<n<\frac{1}{C_1} \left( \frac{k}{n} \right)^2$ as long as $C$ is large enough. (We remark that here, unlike in the proof of Theorem~\ref{thm:lower S_k}, there is nothing to be gained from splitting the point set with vertical lines before applying the cutting lemma, since a line induced by an edge between two points to the right of all lines will still intersect the regions to the left of the lines.)
\end{proof}

From the upper bound in Theorem~\ref{teo:S'} for $\overline\max\ S_k'(K_n)$, it follows that $\overline\max\ e_k'(K_n)\leq O(n+k)$. Just as in the original situation, determining the asymptotic behavior of $\overline\max\ e_k'(K_n)$ appears to considerably more challenging than that of $\overline\max\ S_k'(K_n)$. Still, we have some reasons to believe that improving the upper bound $\overline\max\ e_k'(K_n)\leq O(n+k)$ might be a simpler task than proving Conjecture~\ref{conj:e_k}.

\bibliographystyle{plain}
\bibliography{refs}

\newpage
\section{Appendix (Proof of Theorem \ref{teo:lin e_k})}

\begin{proof}[Proof of Theorem \ref{teo:lin e_k}]
Let $k = (m-2)^2 + r$ where $r \in [1, 2m-3]$, so that $m = \left\lceil \sqrt{k} \right\rceil + 1$. First, we produce configurations of $2m-1$ or $2m$ points which induce $\Omega(m)$ edges with $k$ crossings. One such configuration is depicted in Figure~\ref{fig12}. Afterwards, we string together sufficiently many copies of such configurations so that they do not interfere with each other to obtain a drawing with $\Omega(n)$ edges having exactly $k$ crossings. 

\begin{figure}[H]
\begin{center}
\includegraphics[scale = 0.7]{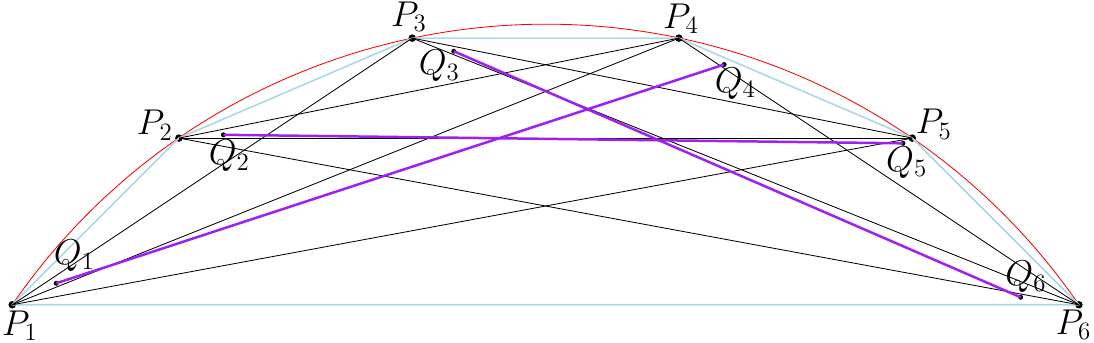}
  \caption{A configuration of $12$ points for $m=6$ and $k = 19$.}
  \label{fig12}
\end{center}
\end{figure}

We assume first that $n$ is even. The fact that $k\leq \left( \frac{n-2}{2} \right)^2$ implies $2m \le n$.

\noindent\textbf{Case 1.} \label{Case1}
If $m = 2j$ (this is the only case that is spelled out in Section~\ref{sec:maxe}), we place $m$ equally spaced vertices $P_1, P_2, \ldots, P_m$ on a small circular arc in clockwise order. We also place $m$ other vertices $Q_1, Q_2, \ldots, Q_m$ such that $Q_i$ is very close to $P_i$ for $i \in [m]$. For the rest of this case, all indices are taken mod $m$. Our goal is to perturb $Q_1, Q_2, \ldots, Q_m$ so that the $j$ \textit{interior diameters}
\begin{align}
\overline{Q_1Q_{j+1}}, \overline{Q_2Q_{j+2}}, \ldots, \overline{Q_jQ_{2j}} \nonumber
\end{align}
are each crossed exactly $k$ times. Notice that $\overline{Q_iQ_{i+j}}$ is crossed by all $(m-2)^2$ edges of the form $\overline{P_aP_b}$, $\overline{Q_aQ_b}$, $\overline{P_aQ_b}$, and $\overline{Q_aP_b}$ for $a \in [i+1, i + j - 1]$ and $b \in [i + j + 1, i + 2j-1]$. Now, we address the edges that are incident to at least one of $P_{i}$ and $P_{i+j}$.

For all $i \in [m]$, we begin by perturbing $Q_1, Q_2, 
\ldots, Q_m$ so that $Q_i$ lies in between $\overrightarrow{P_iP_{i+j}}$ and $\overrightarrow{P_{i-1}P_i}$. This ensures that the $m-1$ rays
\begin{align}\label{equ:rays1'}
\overrightarrow{P_iQ_{i+j-1}}, \overrightarrow{P_iP_{i+j-1}}, \overrightarrow{P_iQ_{i+j-2}}, \overrightarrow{P_iP_{i+j-2}}, \ldots, \overrightarrow{P_iQ_{i+1}}, \overrightarrow{P_iP_{i+1}}, \overrightarrow{P_{i-1}P_i}
\end{align}
appear in counterclockwise order around $P_i$ and that the $m-1$ rays
\begin{align}\label{equ:rays2'}
\overrightarrow{P_{i+j}Q_{i-1}}, \overrightarrow{P_{i+j}P_{i-1}}, \overrightarrow{P_{i+j}Q_{i-2}}, \overrightarrow{P_{i+j}P_{i-2}}, \ldots, \overrightarrow{P_{i+j}Q_{i+j+1}}, \overrightarrow{P_{i+j}P_{i+j+1}}, \overrightarrow{P_{i+j-1}P_{i+j}}
\end{align}
appear in counterclockwise order around $P_{i+j}$. Moreover, because $Q_i$ and $Q_{i+j}$ lie on different sides of $\overline{P_iP_{i+j}}$, it follows that $\overline{P_iP_{i+j}}$ also crosses $\overline{Q_iQ_{i+j}}$. Now, we proceed with a series of more precise secondary perturbations.

If $r \in [1, m-1]$, then we perturb $Q_i$ so that it lies in between $\overrightarrow{P_iP_{i+j}}$ and $\overrightarrow{P_iQ_{i+j-1}}$. Additionally, we perturb $Q_{i+j}$ so that it lies in between the $(r-1)^{\text{th}}$ and $r^{\text{th}}$ rays listed in (\ref{equ:rays2'}), where the $0^{\text{th}}$ ray is defined as $\overrightarrow{P_{i+j}P_i}$. This ensures that the edges incident to exactly one of $P_i$ and $P_{i+j}$ that cross $\overline{Q_iQ_{i+j}}$ are exactly the edges corresponding to the first $r-1$ rays listed in (\ref{equ:rays2'}). Thus, the total number of edges crossing $\overline{Q_iQ_{i+j}}$ is indeed
\begin{align}
(m-2)^2 + 1 + (r-1) =  k. \nonumber
\end{align}
Because these secondary perturbations of $Q_i$ and $Q_{i+j}$ do not alter the crossings of any other interior diameters, we can perturb $Q_i$ and $Q_{i+j}$ in this fashion for all $i \in [j]$ to transform 
\begin{align}
\overline{Q_1Q_{j+1}}, \overline{Q_2Q_{j+2}}, \ldots, \overline{Q_jQ_{2j}} \nonumber
\end{align}
into edges with $k$ crossings. See Figure~\ref{fig12}.

If $r \in [m, 2m-3]$, then we perturb $Q_i$ so that it lies in between the $(r-m+1)^{\text{th}}$ and $(r-m+2)^{\text{th}}$ rays of (\ref{equ:rays1'}). Additionally, we perturb $Q_{i+j}$ so that it lies in between $\overrightarrow{P_{i+j}P_{i+j+1}}$ and $\overrightarrow{P_{i+j-1}P_{i+j}}$. This ensures that edges incident to exactly one of $P_i$ and $P_{i+j}$ that cross $\overline{Q_iQ_{i+j}}$ are the edges corresponding to the first $r-m+1$ rays of (\ref{equ:rays1'}) and the edges corresponding to the first $m-2$ rays of (\ref{equ:rays2'}). Thus, the total number of edges crossing $\overline{Q_iQ_{i+j}}$ is indeed
\begin{align}
(m-2)^2 + 1 + (r-m+1) + (m-2) = (m-2)^2 + r = k. \nonumber
\end{align}
Once again, because these secondary perturbations of $Q_i$ and $Q_{i+j}$ do not alter the crossings of any other interior diameters, we can perturb $Q_i$ and $Q_{i+j}$ in this fashion for all $i \in [j]$ to transform
\begin{align}
\overline{Q_1Q_{j+1}}, \overline{Q_2Q_{j+2}}, \ldots, \overline{Q_jQ_{2j}} \nonumber
\end{align}
into edges with $k$ crossings. See Figure~\ref{fig13}.

\begin{figure}[H]
\begin{center}
\includegraphics[scale = 0.7]{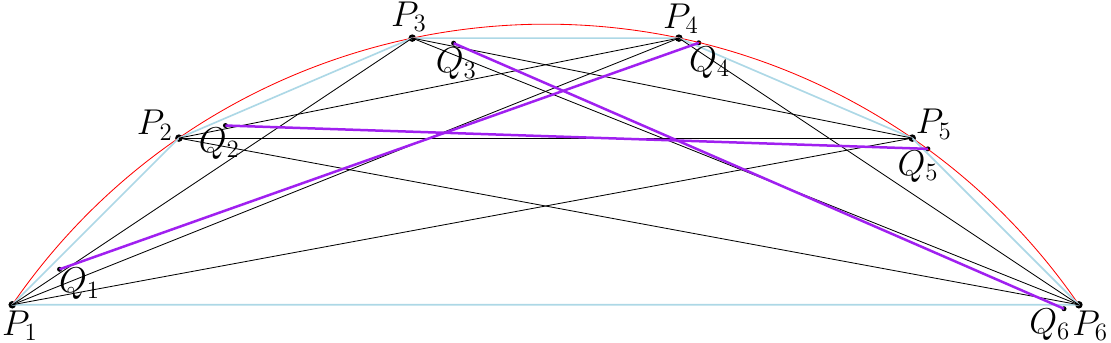}
  \caption{A configuration of $12$ points for $m=6$ and $k = 23$.}
  \label{fig13}
\end{center}
\end{figure}

Thus, we have covered all cases where $n$ is even and $k \in \left( (2j-2)^2, (2j-1)^2 \right]$ for some $j \in \mathbb{Z}^+$.

\noindent\textbf{Case 2.} \label{Case2}
If $m = 2j+1$ and $r \in [2, 2m-4]$, we once again place $m$ equally spaced vertices $P_1, P_2, \ldots, P_m$ on a minor circular arc in clockwise order and construct $m$ vertices $Q_1, Q_2, \ldots, Q_m$ such that $Q_i$ is very close to $P_i$ for $i \in [m]$. For the rest of this case, we also take all indices modulo $m$. As before, our goal is for the $j$ interior diameters 
\begin{align}
\overline{Q_1Q_{j+1}}, \overline{Q_2Q_{j+2}}, \ldots, \overline{Q_{j}Q_{2j}} \nonumber
\end{align}
to each get crossed $k$ times. Observe that $\overline{Q_iQ_{i+j}}$ is crossed by all $(m-1)(m-3)$ edges of the form $\overline{P_aP_b}$, $\overline{Q_aQ_b}$, $\overline{P_aQ_b}$, and $\overline{Q_aP_b}$ for $a \in [i+1, i + j - 1]$ and $b \in [i + j + 1, i + 2j]$. Now, we address edges that are incident to at least one of $P_{i}$ and $P_{i+j}$.
For all $i \in [m]$, we begin by perturbing $Q_1, Q_2, 
\ldots, Q_m$ so that $Q_i$ lies in between $\overrightarrow{P_iP_{i+j}}$ and $\overrightarrow{P_{i-1}P_i}$. This ensures that the $m-2$ rays
\begin{align}\label{equ:rays3}
\overrightarrow{P_iQ_{i+j-1}}, \overrightarrow{P_iP_{i+j-1}}, \overrightarrow{P_iQ_{i+j-2}}, \overrightarrow{P_iP_{i+j-2}}, \ldots, \overrightarrow{P_iQ_{i+1}}, \overrightarrow{P_iP_{i+1}}, \overrightarrow{P_{i-1}P_i}
\end{align}
appear in counterclockwise order with respect to $P_i$ and that the $m$ rays
\begin{align}\label{equ:rays4}
\overrightarrow{P_{i+j}Q_{i-1}}, \overrightarrow{P_{i+j}P_{i-1}}, \overrightarrow{P_{i+j}Q_{i-2}}, \overrightarrow{P_{i+j}P_{i-2}}, \ldots, \overrightarrow{P_{i+j}Q_{i+j+1}}, \overrightarrow{P_{i+j}P_{i+j+1}}, \overrightarrow{P_{i+j-1}P_{i+j}}
\end{align}
appear in counterclockwise order with respect to $P_{i+j}$. Moreover, because $Q_i$ and $Q_{i+j}$ lie on different sides of $\overline{P_iP_{i+j}}$, it follows that $\overline{P_iP_{i+j}}$ also crosses $\overline{Q_iQ_{i+j}}$. Now, we proceed with a series of more precise secondary perturbations.

If $r \in [2, m-1]$, then we perturb $Q_i$ so that it lies in between $\overrightarrow{P_iP_{i+j}}$ and $\overrightarrow{P_iQ_{i+j-1}}$. Additionally, we perturb $Q_{i+j}$\footnote[1]{The condition given in the initial perturbation implies that $Q_{i+j}$ must lie in between $\overrightarrow{P_{i+j}P_{i-1}}$ and $\overrightarrow{P_{i+j-1}P_{i+j}}$, and this stipulation forces $r \ge 2$ to hold for this perturbation of $Q_{i+j}$ to be valid.} so that it lies in between the $r^{\text{th}}$ and $(r+1)^{\text{th}}$ rays of (\ref{equ:rays4}). This ensures that the edges incident to exactly one of $P_i$ and $P_{i+j}$ that cross $\overline{Q_iQ_{i+j}}$ are the edges corresponding to the first $r$ rays of (\ref{equ:rays4}). Thus, the total number of edges crossing $\overline{Q_iQ_{i+j}}$ is indeed
\begin{align}
(m-1)(m-3) + 1 + r = (m-2)^2 + r = k. \nonumber
\end{align}
Because these secondary perturbations of $Q_i$ and $Q_{i+j}$ do not alter the crossings of any other interior diameters, we can perturb $Q_i$ and $Q_{i+m}$ in this fashion for all $i \in [j]$ to transform 
\begin{align}
\overline{Q_1Q_{j+1}}, \overline{Q_2Q_{j+2}}, \ldots, \overline{Q_{j}Q_{2j}} \nonumber
\end{align}
into edges with $k$ crossings.

If $r \in [m, 2m-4]$, then we perturb $Q_i$ so that it lies in between the $(r-m+1)^{\text{th}}$ and $(r-m+2)^{\text{th}}$ rays of (\ref{equ:rays3}). Additionally, we perturb $Q_{i+j}$ so that it lies in between $\overrightarrow{P_{i+j}P_{i+j+1}}$ and $\overrightarrow{P_{i+j-1}P_{i+j}}$. This ensures that the set of edges incident to exactly one of $P_i$ and $P_{i+j}$ that cross $\overline{Q_iQ_{i+j}}$ are the edges corresponding to the first $(r-m+1)$ rays of (\ref{equ:rays3}) and the edges corresponding to the first $m-1$ rays of (\ref{equ:rays4}). Thus, the total number of edges crossing $\overline{Q_iQ_{i+j}}$ is indeed
\begin{align}
(m-1)(m-3) + 1 + (r-m+1) + (m-1) = (m-2)^2 + r = k. \nonumber
\end{align}
Once again, because these secondary perturbations of $Q_i$ and $Q_{i+j}$ do not alter the crossings of any other interior diameters, we can perturb $Q_i$ and $Q_{i+m}$ in this fashion for all $i \in [j]$ to transform 
\begin{align}
\overline{Q_1Q_{j+1}}, \overline{Q_2Q_{j+2}}, \ldots, \overline{Q_{j}Q_{2j}} \nonumber
\end{align}
into edges with $k$ crossings.

\noindent\textbf{Case 3.}
If $m = 2j+1$ and $r = 1$, then we start with the previously presented construction consisting of $2m-2$ points for edges with $(m-3)^2 + (m-2)$ crossings\footnote[2]{Because $m-1$ is even and $(m-3)^2 + (m-2) \in \left((m-3)^2, (m-2)^2 \right]$ for $m \ge 3$, it follows that this construction was already covered in Case 1.} and turn it into a configuration consisting of $2m-1$ points that has $\Omega(m)$ edges with $(m-2)^2 + 1$ crossings.

For the sake of readability, we reproduce a description of this previously presented construction. Place $m-1$ equally spaced vertices $P_1, P_2, \ldots, P_{m-1}$ on a small circular arc in clockwise order and construct $m-1$ vertices $Q_1, Q_2, \ldots, Q_{m-1}$ such that $Q_i$ is very close to $P_i$ for $i \in [m-1]$. Henceforth, take all indices modulo $m-1$. For all $i \in [j]$, we know $Q_i$ lies in between $\overrightarrow{P_iP_{i+j}}$ and $\overrightarrow{P_{i}Q_{i+j-1}}$, while $Q_{i+j}$ lies in between $\overrightarrow{P_{i+j}P_{i+j+1}}$ and $\overrightarrow{P_{i+j-1}P_{i+j}}$. 

\begin{figure}[H]
\begin{center}
\includegraphics[scale = 0.7]{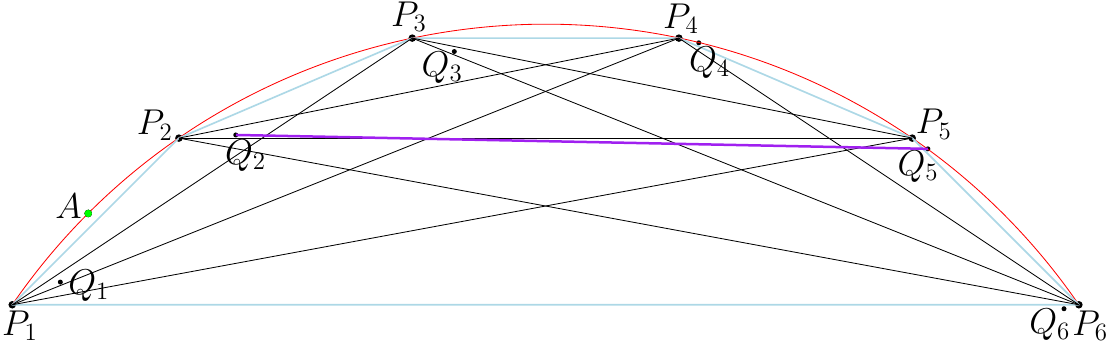}
  \caption{A configuration of $13$ points for $m = 7$ and $k = 26$.}
  \label{fig14}
\end{center}
\end{figure}

Now, we add a vertex $A$ at the midpoint of the arc between $P_1$ and $P_2$ and consider how many edges incident to $A$ cross the $j-2$ interior diameters
\begin{align}
\overline{Q_2Q_{j+2}}, \overline{Q_3Q_{j+3}}, \ldots, \overline{Q_{j-1}Q_{2j-1}}. \nonumber
\end{align}
First of all, for each $i \in [2, j-1]$, the $m-2$ points on the opposite side of $\overleftrightarrow{Q_iQ_{i+j}}$ with respect to $A$ are $P_{i+1}, Q_{i+1}, P_{i+2}, Q_{i+2}, \ldots, P_{i+j-1}, Q_{i+j-1}, P_{i+j}$, and the $m-3$ edges
\begin{align}
\overline{AP_{i+1}}, \overline{AQ_{i+1}}, \overline{AP_{i+2}}, \overline{AQ_{i+2}}, \ldots, \overline{AP_{i+j-1}}, \overline{AQ_{i+j-1}} \nonumber
\end{align}
clearly cross $\overline{Q_iQ_{i+j}}$. Next, we focus on $\overline{AP_{i+j}}$. For any $i \in [2, j-1]$, it is easy to see $A$ lies in between $\overrightarrow{P_{i+j}P_i}$ and $\overrightarrow{P_{i+j}P_{i+j+1}}$. The characterizations of $Q_i$ and $Q_{i+j}$ given above imply that the angle $\angle Q_iP_{i+j}Q_{i+j}$ fully contains $\angle P_iP_{i+j}P_{i+j+1}$, whence $A$ lies in between $\overrightarrow{P_{i+j}Q_i}$ and $\overrightarrow{P_{i+j}Q_{i+j}}$. This means $\overline{AP_{i+j}}$ does indeed cross $\overline{Q_iQ_{i+j}}$, so the number of edges crossing each of the $j-2$ interior diameters
\begin{align}
\overline{Q_2Q_{j+2}}, \overline{Q_3Q_{j+3}}, \ldots, \overline{Q_{j-1}Q_{2j-1}} \nonumber
\end{align}
is precisely $(m-3)^2 + (m-2) + (m-2) = (m-2)^2 + 1$, as required. See Figure~\ref{fig14}.

\noindent\textbf{Case 4.} 
Lastly, we consider $m = 2j+1$ and $r = 2m-3$, which is equivalent to $k = (m-1)^2$. Here, it suffices to place $2m$ vertices $P_1, P_2, \ldots, P_{2m}$ on a circular arc in clockwise order, as in this case each of
\begin{align}
\overline{P_1P_{m+1}}, \overline{P_2P_{m+2}}, \ldots, \overline{P_mP_{2m}} \nonumber
\end{align}
has exactly $k = (m-1)^2$ crossings.

Thus, we have covered all cases where $n$ is even and $k \in \left( (2j-1)^2, (2j)^2 \right]$ for some $j \in \mathbb{Z}^+$.

\noindent\textbf{Case 5.} 
Now, assume that $n$ odd and $k \in \left[1, \left( \frac{n-3}{2} \right)^2 \right]$. Because $2m < n$ holds for $k$ in this interval, we can utilize the appropriate constructions presented above to once again obtain configurations of $2m-1$ or $2m$ points containing $\Omega(m)$ edges with $k$ crossings. 

For all $n$ and $k$ already addressed above, we now describe how $t = \left\lfloor \frac{n}{2m} \right\rfloor$ copies of this configuration of $2m$ or $2m-1$ points can be arranged together so that no additional edges in the drawing interfere with the interior diameters of the individual configurations. Place points $A_1, B_1, A_2, B_2, \ldots, A_t, B_t$ on a circle $C$ in clockwise order (we remark that none of these $2t$ points will be vertices of ${\mathcal{D}}_n$). For each $\overline{A_iB_i}$, we construct a minor circular arc $c_i$ containing $A_i$ and $B_i$ that is fully contained inside $C$, as seen in Figure~\ref{fig15}, and create a configuration of $2m-1$ or $2m$ points using $c_i$ as the minor circular arc. Finally, place any leftover vertices in a cluster near the center of $C$.

We can make each $c_i$ sufficiently flat so as to ensure that no interior diameter of the configuration based at $c_i$ is intersected by additional edges of the drawing. Thus, each of these interior diameters still has $k$ crossings, and we have that
\begin{align}
e_k(\mathcal D) \ge t \cdot \Omega(m) = \left\lfloor \frac{n}{2m} \right\rfloor \cdot \Omega(m) = \Omega(n) \nonumber.
 \end{align}

\begin{figure}[H]
\begin{center}
\includegraphics[scale = 0.45]{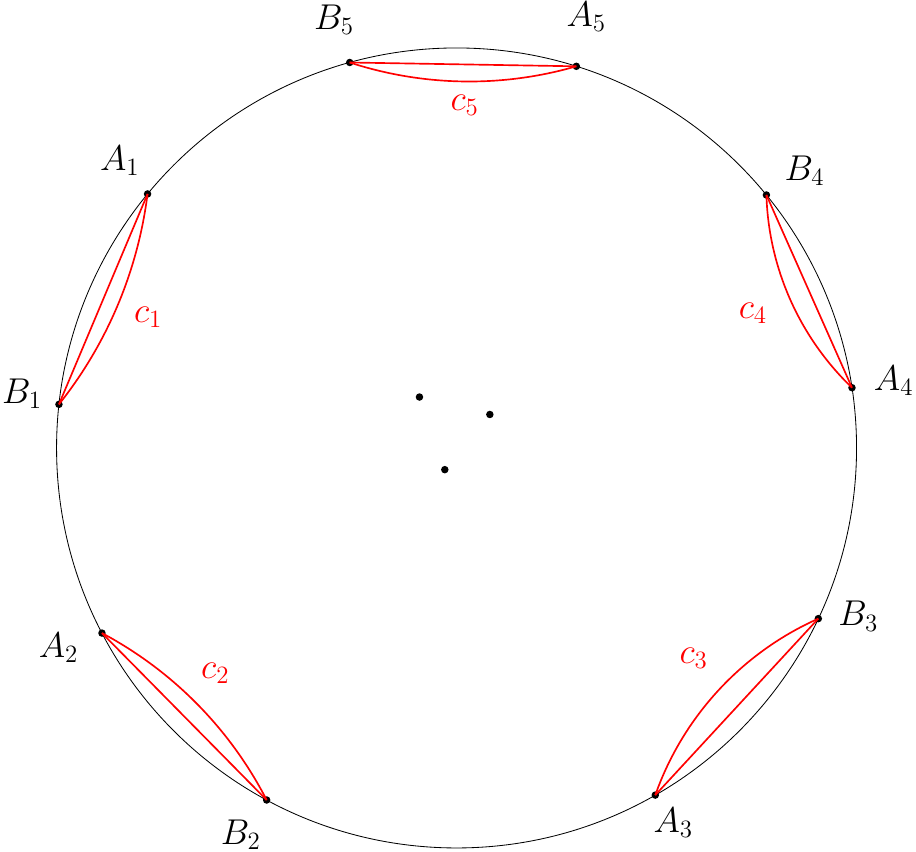}
  \caption{Grouping together $5$ copies of the $2m$-point configuration described above. Some leftover nodes have been placed near the center of the circle.}
  \label{fig15}
\end{center}
\end{figure}

Lastly, we present constructions for $n$ odd and $k \in \left( \left( \frac{n-3}{2} \right)^2, \left\lfloor \left( \frac{n-2}{2} \right)^2 \right\rfloor \right]$, as $2m \le n$ does not hold in this case.

\noindent\textbf{Case 6.} 
If $\frac{n-3}{2}$ is odd, then set $m = \frac{n-1}{2}$ and write $k = (m-2)^2 + 2(m-1) + d$ with $d \in [0, m-2]$.

We start with the previously presented construction from Case 1 consisting of $2m$ points for edges with $(m-2)^2 + (m-1) + d$ crossings, noting that $(m-2)^2 + (m-1) + d \in \left((m-2)^2, (m-1)^2 \right]$. We transform this configuration into one with $2m+1$ points that has $\Omega(m)$ edges with $k = (m-2)^2 + 2(m-1) + d$ crossings. Noting $m$ is even, set $m = 2j$.

\begin{figure}[H]
\begin{center}
\includegraphics[scale = 0.7]{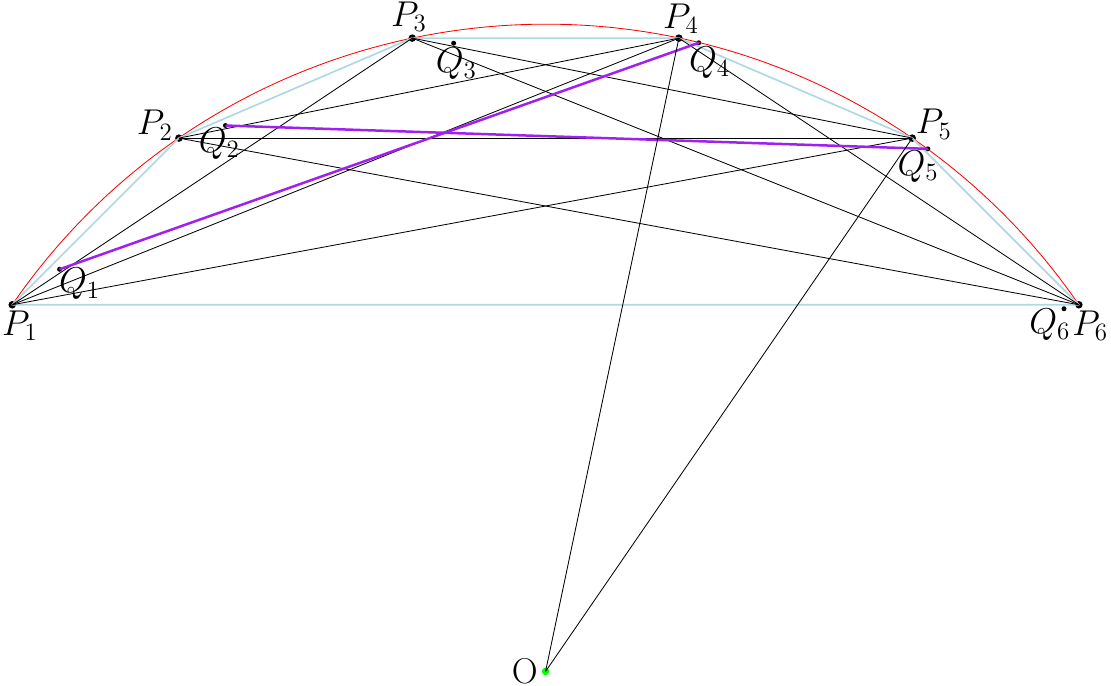}
  \caption{A configuration of $13$ points for $m = 6$ and $k = 28$.}
  \label{fig16}
\end{center}
\end{figure}

We add the center $O$ of the circle containing $P_1,\ldots, P_m$ and consider crossings between edges incident to $O$ and the $j-1$ interior diameters
\begin{align}
\overline{Q_1Q_{j+1}}, \overline{Q_2Q_{j+2}}, \ldots, \overline{Q_{j-1}Q_{2j-1}} \nonumber
\end{align}
from the original construction. For each $i \in [j-1]$, we know that the $m-1$ points on the opposite side of $\overleftrightarrow{Q_iQ_{i+j}}$ with respect to $O$ are $P_{i+1}, Q_{i+1}, P_{i+2}, Q_{i+2}, \ldots, P_{i+j-1}, Q_{i+j-1}, P_{i+j}$. Furthermore, the $m-2$ edges
\begin{align}
\overline{OP_{i+1}}, \overline{OQ_{i+1}}, \overline{OP_{i+2}}, \overline{OQ_{i+2}}, \ldots, \overline{OP_{i+j-1}}, \overline{OQ_{i+j-1}} \nonumber
\end{align}
clearly cross $\overline{Q_iQ_{i+j}}$. Now, we focus on $\overline{OP_{i+j}}$. For any $i \in [j-1]$, it is easy to see that $O$ lies in between $\overrightarrow{P_{i+j}P_i}$ and $\overrightarrow{P_{i+j}P_{i+j+1}}$ since $\angle P_iP_{i+j}P_{i+j+1}$ is obtuse. Now, the characterizations of $Q_i$ and $Q_{i+j}$ given in the previously presented construction imply that $\angle Q_iP_{i+j}Q_{i+j}$ fully contains $\angle P_iP_{i+j}P_{i+j+1}$, whence $O$ lies in between $\overrightarrow{P_{i+j}Q_i}$ and $\overrightarrow{P_{i+j}Q_{i+j}}$. This means $\overline{OP_{i+j}}$ crosses $\overline{Q_iQ_{i+j}}$, so the number of edges crossing each of the $j-1$ interior diameters
\begin{align}
\overline{Q_1Q_{j+1}}, \overline{Q_2Q_{j+2}}, \ldots, \overline{Q_{j-1}Q_{2j-1}} \nonumber
\end{align}
is precisely $\left( (m-2)^2 + (m-1) + d \right) + (m-1) = (m-2)^2 + 2(m-1) + d = k$, as required. See Figure~\ref{fig16}.

\noindent\textbf{Case 7.} 
If $\frac{n-3}{2}$ is even and $k \in \left[ \left( \frac{n-3}{2} \right)^2 + 2, \left\lfloor \left( \frac{n-2}{2} \right)^2 \right\rfloor \right]$, then set $m = \frac{n-1}{2}$ and write $k = (m-2)^2 + (m-1) + m + d$ with $d \in [0, m-3]$. 

We start with the construction from Case 2 consisting of $2m$ points for edges with $(m-2)^2 + (m-1) + d$ crossings, noting that $(m-2)^2 + (m-1) + d \in \left((m-2)^2, (m-1)^2 \right]$, and turn it into a configuration of $2m+1$ points that has $\Omega(m)$ edges with $k = (m-2)^2 + (m-1) + m + d$ crossings. Noting $m$ is odd, set $m = 2j+1$.

Place $m$ equally spaced vertices $P_1, P_2, \ldots, P_m$ on a circular arc in clockwise order and construct $m$ vertices $Q_1, Q_2, \ldots, Q_m$ such that $Q_i$ is sufficiently close to $P_i$ for $i \in [m]$. Henceforth, take all indices modulo $m$. For all $i \in [m]$, we perturb the points so that $Q_i$ lies in between $\overrightarrow{P_iP_{i+j}}$ and $\overrightarrow{P_{i-1}P_i}$. Our goal is to ensure the $j-1$ interior diameters
\begin{align}
\overline{Q_1Q_{j+2}}, \overline{Q_2Q_{j+3}}, \ldots, \overline{Q_{j-1}Q_{2j}} \nonumber
\end{align} 
each have $k$ crossings.

For every $i \in [j+2, 2j]$, we perturb so that $Q_i$ lies in between $\overrightarrow{P_iP_{i+1}}$ and $\overrightarrow{P_{i-1}P_i}$ and $Q_{i+j}$ lies in between the $(d+2)^{\text{th}}$ and $(d+3)^{\text{th}}$ rays of (\ref{equ:rays4}). Using reasoning similar to that from Case 2, we can now deduce that the number of edges from the complete graph on the $2m$ existing vertices that cross each of the $j-1$ interior diameters
\begin{align}
\overline{Q_1Q_{j+2}}, \overline{Q_2Q_{j+3}}, \ldots, \overline{Q_{j-1}Q_{2j}} \nonumber
\end{align}
is precisely
\begin{align}
(m-1)(m-3) + 1 + (m-3) + (d+2) = (m-2)^2 + (m-1) + d. \nonumber
\end{align}

Now, we add the center $O$ of the circle containing $P_1,\ldots,P_n$ and consider crossings between edges incident to $O$ and the $j-1$ interior diameters
\begin{align}
\overline{Q_1Q_{j+2}}, \overline{Q_2Q_{j+3}}, \ldots, \overline{Q_{j-1}Q_{2j}}. \nonumber
\end{align}
For every $i \in [j+2, 2j]$, we know $P_{i+j+1}, Q_{i+j+1}, P_{i+j+2}, Q_{i+j+2}, \ldots, P_{i-1}, Q_{i-1}, P_i$ are the $m$ points on the opposite side of $\overleftrightarrow{Q_iQ_{i+j}}$ with respect to $O$. Furthermore, the $m-1$ edges
\begin{align}
\overline{OP_{i+j+1}}, \overline{OQ_{i+j+1}}, \overline{OP_{i+j+2}}, \overline{OQ_{i+j+2}}, \ldots, \overline{OP_{i-1}}, \overline{OQ_{i-1}} \nonumber
\end{align}
clearly cross $\overline{Q_iQ_{i+j}}$. Next, we focus on $\overline{OP_i}$. For any $i \in [j+2, 2j]$, it is easy to see that $O$ lies in between $\overrightarrow{P_iP_{i+j}}$ and $\overrightarrow{P_iP_{i+1}}$ since $\angle P_{i+j}P_iP_{i+1}$ is obtuse. Now, the characterizations of $Q_i$ and $Q_{i+j}$ presented in this case imply that $\angle Q_{i+j}P_iQ_i$ fully contains $\angle P_{i+j}P_iP_{i+1}$, whence $O$ lies in between $\overrightarrow{P_iQ_{i+j}}$ and $\overrightarrow{P_iQ_i}$. This means $\overline{OP_i}$ does indeed cross $\overline{Q_iQ_{i+j}}$, so the number of edges crossing each of the $j-1$ interior diameters
\begin{align}
\overline{Q_1Q_{j+1}}, \overline{Q_2Q_{j+2}}, \ldots, \overline{Q_{j-1}Q_{2j-1}} \nonumber
\end{align}
is precisely $\left( (m-2)^2 + (m-1) + d \right) + m = k$, as required.

\noindent\textbf{Case 8.} 
If $\frac{n-3}{2}$ is even and $k = \left( \frac{n-3}{2} \right)^2 + 1$, we set $m = \frac{n-1}{2}$, which means $k = (m-1)^2 + 1$. 

We take the previously presented construction from Case 2 consisting of $2m$ points for edges with $(m-2)^2 + m$ crossings and turn it into a configuration of $2m+1$ points that has $\Omega(m)$ edges with $k = (m-1)^2 + 1$ crossings. Noting $m$ is odd, set $m = 2j+1$.

We add the center $O$ of the circle containing $P_1,\ldots,P_n$ and consider crossings between edges incident to $O$ and the $j$ interior diameters
\begin{align}
\overline{Q_1Q_{j+1}}, \overline{Q_2Q_{j+2}}, \ldots, \overline{Q_jQ_{2j}} \nonumber
\end{align}
from the previously presented construction. For each $i \in [j]$, we know the $m-2$ points on the opposite side of $\overleftrightarrow{Q_iQ_{i+j}}$ with respect to $O$ are $P_{i+1}, Q_{i+1}, P_{i+2}, Q_{i+2}, \ldots, P_{i+j-1}, Q_{i+j-1}, P_{i+j}$. Furthermore, the $m-3$ edges
\begin{align}
\overline{OP_{i+1}}, \overline{OQ_{i+1}}, \overline{OP_{i+2}}, \overline{OQ_{i+2}}, \ldots, \overline{OP_{i+j-1}}, \overline{OQ_{i+j-1}} \nonumber
\end{align}
clearly cross $\overline{Q_iQ_{i+j}}$. Now, we focus on $\overline{OP_{i+j}}$. For any $i \in [j]$, it is easy to see that $O$ lies in between $\overrightarrow{P_{i+j}P_i}$ and $\overrightarrow{P_{i+j}P_{i+j+1}}$ since $\angle P_iP_{i+j}P_{i+j+1}$ is obtuse. Now, the characterizations of $Q_i$ and $Q_{i+j}$ given in the previously presented construction imply that $\angle Q_iP_{i+j}Q_{i+j}$ fully contains $\angle P_iP_{i+j}P_{i+j+1}$, whence $O$ lies in between $\overrightarrow{P_{i+j}Q_i}$ and $\overrightarrow{P_{i+j}Q_{i+j}}$. This means that $\overline{OP_{i+j}}$ crosses $\overline{Q_iQ_{i+j}}$, so the number of edges crossing each of the $j$ interior diameters
\begin{align}
\overline{Q_1Q_{j+1}}, \overline{Q_2Q_{j+2}}, \ldots, \overline{Q_jQ_{2j}} \nonumber
\end{align}
is precisely $\left( (m-2)^2 + m \right) + (m-2) = (m-1)^2 + 1 = k$, as desired.

Thus, we have also covered $k \in \left( \left( \frac{n-3}{2} \right)^2, \left\lfloor \left( \frac{n-2}{2} \right)^2 \right\rfloor \right]$ for $n$ odd, exhausting all required cases.
\end{proof}

\end{document}